\documentclass[11pt,oneside]{amsart}
\usepackage[T1]{fontenc}
\usepackage{marginnote}
\usepackage[marginparwidth=2.3cm]{geometry}
\usepackage{amsmath}
\usepackage{amsthm}
\usepackage{amssymb}
\usepackage{amsfonts}
\usepackage{enumerate}
\usepackage{esvect}
\usepackage{xcolor}
\usepackage{esint}
\usepackage{bbold}
\usepackage{float}

\usepackage{verbatim}

\newcommand{\overbar}[1]{\mkern 1.5mu\overline{\mkern-1.5mu#1\mkern-1.5mu}\mkern 1.5mu}

\newtheorem{theorem}{Theorem}
\newtheorem{lemma}{Lemma}
\newtheorem{cor}{Corollary}

\newcommand{\dd}{\mathrm{d}}

\newcommand{\eps}{\epsilon}

\newcommand{\mres}{\mathbin{\vrule height 1.6ex depth 0pt width
0.13ex\vrule height 0.13ex depth 0pt width 1.3ex}}
\newcommand{\Frac}[2]{\displaystyle \frac{#1}{#2}}

\newcommand{\Div}{\, \nabla \cdot}

\newcommand{\bb}{\mathbf{b}}
\newcommand{\be}{\mathbf{e}}
\newcommand{\bulk}{Q_L}

\newcommand{\infbnd}{\underline{Q_L}}

\newcommand{\R}{\mathbb R}
\newcommand{\N}{\mathbb N}

\newcommand{\shift}{\beta}
\newcommand{\per}{\text{per}}
\newcommand{\sign}{\text{sign}\,}
\newcommand{\Aper}{\mathcal{A}^{per}(Q_L)}

\newcommand{\dist}{\mathrm{dist}}

\DeclareSymbolFont{symbolsC}{U}{pxsyc}{m}{n}
\DeclareMathSymbol{\defeq}{\mathrel}{symbolsC}{"42}

\def\Xint#1{\mathchoice
{\XXint\displaystyle\textstyle{#1}}%
{\XXint\textstyle\scriptstyle{#1}}%
{\XXint\scriptstyle\scriptscriptstyle{#1}}%
{\XXint\scriptscriptstyle\scriptscriptstyle{#1}}%
\!\int}
\def\XXint#1#2#3{{\setbox0=\hbox{$#1{#2#3}{\int}$ }
\vcenter{\hbox{$#2#3$ }}\kern-.6\wd0}}

\def\dashint{\Xint-}

\DeclareSymbolFont{symbolsC}{U}{pxsyc}{m}{n}
\DeclareMathSymbol{\coloneqq}{\mathrel}{symbolsC}{"42}

\begin{document}
\title[Uniform energy distribution]{Uniform energy distribution in a pattern-forming system of surface charges}
\author{Katarina Bellova}
\address[K.~Bellova]{Institute of Mathematics, Leipzig University, Augustusplatz 10, 04109 Leipzig, Germany}
\email{bellova@math.uni-leipzig.de}
\author{Antoine Julia}
\address[A.~Julia]{Dipartimento di Matematica ``T.~Levi-Civita'', via Trieste 63, 35121 Padova, Italy.}
\email{ajulia@math.unipd.it}
\author{Felix Otto}
\address[F.~Otto]{Max Planck Institute for Mathematics in the Sciences, Inselstra{\ss}e 22, 04103 Leipzig, Germany}
\email{otto@mis.mpg.de}
\thanks{
	A.J.~has been supported
	by University of Padova STARS Project ``Sub-Riemannian Geometry and Geometric Measure Theory Issues: Old and New'';
	and by the INdAM -- GNAMPA Project 2019 ``Rectifiability in Carnot groups''.
      }
\date{\today}

\begin{abstract}
We consider a variational model for a charge density $u\in\{-1,1\}$
on a (hyper)plane, with a short-range attraction coming from the interfacial energy 
and a long-range repulsion coming from the electrostatic energy. This competition
leads to pattern formation. We prove that the interfacial energy density is
(asymptotically) equidistributed at scales large compared to the scale of the pattern. We follow
the strategy laid out in \cite{ACO09}. The challenge
comes from the reduced screening capabilities of surface charges compared
to the volume charges considered in \cite{ACO09}.
\end{abstract}
\maketitle

\tableofcontents

\section{Introduction}

\medskip

The interplay of short-range attraction and long-range repulsion selects a
length scale and typically leads to pattern formation. In a thermodynamic limit,
provided the influence of boundary conditions fades away, this competition often seems
to favor a periodic pattern, like equidistant stripes. Within variational models,
periodicity has been established
in 1-d situations, eg.~\cite{Muller1993}, in 2-d models that are not too
far from a packing problem, foremost \cite{Theil2006CMP} but also
\cite{BoPeTh2014CMP}, or in models
that feature a strong anisotropy, eg.~\cite{GiulMull2012CMP}. 

\medskip

In a multi-dimensional
isotropic setting, even in a variational model, a proof of periodicity seems out of reach.
However, robust strategies for proving a uniform distribution of eg.~the energy density 
at scales large compared to the intrinsic scale are available, see \cite{ACO09}.
That paper deals with the popular model where the configuration space consists
of characteristic functions $u\in\{-1,1\}$, the short-range attractive interaction
is the interfacial energy between the two phases, and the long-range repulsive interaction is electrostatic,
with the order parameter $u$ playing the role of a charge density. 
While in \cite{ACO09}, it is the volume {\it fraction} that is implicitly prescribed, 
the problem of prescribed volume is also of interest because of the potential loss of tightness, 
eg.~\cite{KnupMura2013CMP,LuOtto2014CPAM}, 
and its isoperimetric aspects, eg.~\cite{CicaSpad2013CPAM}; 
both problems are related in the regime of small volume fraction, 
eg.~\cite{ChokPele2010SIAM}.

\medskip

The main challenge in establishing
a mesoscopically uniform energy distribution (which by a virial argument also
leads to  equipartition and thus uniform distribution of the interfacial energy) lies in capturing screening 
effects: On mesoscopic scales, charges arrange themselves in such a way as to reduce
the macroscopic part of the electric field $\bb$ as much as possible. 
In this paper, we consider a very similar problem, which however is of different
dimensionality: While the field $\bb$ extends into $d$-dimensional space $\mathbb{R}^d$,
the charges $u$ are (hyper)surface charges constrained to $\mathbb{R}^{d-1}\times\{0\}$.
In other words, we replace the squared $\dot{\mathrm{H}}^{-1}$-norm of $u$ by its fractional
counterpart $\dot{\mathrm{H}}^{-\frac{1}{2}}$. This additional, dimensional, restriction of 
the charge density $|u|\leq 1$ clearly hinders screening and thus poses an additional
challenge. In fact, it turns out that the arguments here, while following the same
strategy, are more involved than in \cite{ACO09}, as we shall detail below. In fact,
some aspects are quite reminiscent of the deep work \cite{Conti2000CPAM} on self-similarity
of twin branching near an Austenite-Martensite interface. Incidentally, our setting of a {\it thin} (partial) conductor, leading to a field energy 
in form of a (squared) {\it fractional} Sobolev norm, resonates with the very active area
of fractional elliptic equations and of thin obstacle problems; let us mention 
\cite{CaffarelliSilvestre2007} as a popular reference.

\medskip

The natural idea from statistical mechanics of establishing the
negligibility of boundary effects by comparing different boundary conditions has been
employed in \cite{ACO09}, with the vanishing flux boundary condition being good for pasting,
and the free boundary suitable for cutting. Incidentally, the ensuing monotonicity
properties, which have been subsequently used for more subtle ferromagnetic pattern formation in
\cite[Lemmas 4 and 5]{OttoVieh2010CalcVar}, 
were also crucial in recent progress on quantitative stochastic homogenization
\cite{ArmsSmar2016AnnENS}. 
Electrostatic screening also plays a role in the popular (mostly two-dimensional)
models for Coulomb gases, and similar arguments have been used in this more subtle
context, e.g.~see the recent \cite{ArmsSerf2019Local}.
Let us also mention that the treatment of boundary layers with incomplete
screening of \cite{ACO09} motivated a variational approach to the regularity of optimal transportation
\cite{GoldOtto2017variational}, see in particular Proposition 3.3 and Lemmas 2.3 and 2.4 therein. 

\medskip

We now give a short summary of this paper, contrasting it with \cite{ACO09}.
We follow  \cite{ACO09} in the sense that a first main step is an energy bound on mesoscopic
scales (Theorem \ref{uniform_bound} here, Lemma 3.5 there). We also follow  \cite{ACO09} in introducing a relaxed
problem (meaning that the non-convex constraint $u\in\{-1,1\}$ is replaced by $u\in[\,-1,1\,]$),
using its dual formulation (Lemma \ref{dual} here, Lemma 3.3 (a) there), appealing to
a trace estimate (Lemma \ref{interpolation} here, Lemma 3.2 there), in order to obtain a
non-linear estimate (Lemma \ref{nonlinestimate} here, Lemma 3.3 (b) there). 
However, and this is the major complication, in this paper, even if we were to completely
neglect the constraint on $u$, the field $\bb$ still would not vanish. In fact, it leads to
what we call the ``over-relaxed problem''. Hence the simple ODE argument that leads to
Lemma 3.5 in \cite{ACO09} has to be replaced by a Campanato iteration (Lemma \ref{iter} here), adjusting 
``shifts'' of the field at every scale (Lemma \ref{itstep} using also Lemma \ref{shiftnonlin}). The similarities with 
\cite[Theorem 2.1, (2.2)]{Conti2000CPAM} ``reverse bootstrap'' are here.

\medskip

As mentioned, we follow  \cite{ACO09} in comparing free, no-flux, and periodic
boundary conditions with their easy relations (Lemma \ref{basic} here, Lemma 3.1 there).
Because of the fact that the relaxed problem (and even the over-relaxed problem)
has a non-vanishing charge density, the pasting of a mesoscopically optimal
pattern into the relaxed solution (Lemma 3.9 in  \cite{ACO09}) is more involved.
It requires an estimate of this over-relaxed charge density (Lemma \ref{pointwise_estimate} here, 
which is the main output of the regularity Section \ref{section_pointwise}), 
and a finer modulation of the mesoscopically optimal pattern 
(based on Lemma \ref{lemma_flow}, which is folklore).
Only this yields Lemmas \ref{sigma_null_sigma} and \ref{sigma_sigma_null}, which relate the mesoscopically localized energy
to the one of the corresponding over-relaxed problem. By yet another Campanato-type
iteration, both lemmas finally imply the negligibility of the energy of the over-relaxed
problem (Lemma \ref{control-v_0}), and thus the main result (Theorem \ref{final_convergence}), from which derives the equipartition between the two contributions to the energy (Theorem \ref{equipartition}).

\medskip

The next sections contain the mathematical formulation of the problem (Section \ref{basics}), the statements of the main theorems (Section \ref{state-main}), and three series of intermediate results on the way to Theorem \ref{final_convergence} (Sections \ref{section_first_bound}, \ref{section_pointwise}, and \ref{section-last}). All the proofs are postponed to Section~\ref{section_Proofs}. 
\section{The problem, notations and definitions} \label{basics}

In the ambient space $\R^d$ with canonical basis $(\be_1,\dots,\be_d)$, we consider two chemical species distributed on the hyperplane $\R^{d-1}\times\{0\}$. These species have different charge densities, renormalised here as $+1$ and $-1$, respectively. The two different species also have a chemical interaction, we model this by introducing an energy term proportional to the interface area between them. In mathematical terms, the charge density is given by a function of locally bounded variation $u$ defined on the hyperplane and taking values $\pm 1$. The interfacial energy is then proportional to the total variation of $u$ (the semi-norm defining the space BV, see Section 3.1 in \cite{AmFuPa_book}). The charges carried on the hyperplane generate an electric field $\bb$ in the whole of $\R^d$,
 satisfying $\nabla \cdot \bb = u\mathcal{H}^{d-1}\mres (\R^{d-1}\times\{0\})$ in the distributional sense. The electric energy is proportional to the square of the 
$\mathrm{L}^2$-norm of $\bb$. 

\medskip

We work in a large cube $Q_L \defeq \left(-L/2,L/2\right)^{d-1} \times (0,L)$, the bottom face of which we denote by $\infbnd = \partial Q_L\cap \{x:x_d=0 \}.$ In the same way as in \cite{ACO09}, we will study the minimizers $(u,\bb)\in \mathrm{BV}(\infbnd,\pm 1) \times \mathrm{L}^2(Q_L,\R^d)$ of the energy
\begin{equation}\label{functional}
  E(u, \bb, Q_L) \coloneqq  \int_{\infbnd} \vert \nabla u \vert+ \int_{Q_L}   \Frac{1}{2}\vert \bb \vert^2,
\end{equation}
under the following constraint, understood in the sense of distributions
\begin{equation}\label{eq:constraints}
  \begin{cases}
    \bb\cdot \be_d &= u \text{ on } \infbnd,\\
    \nabla \cdot \bb &= 0 \text{ in } Q_L,
  \end{cases}
\end{equation}
where, letting $\nu$ be the \textit{inner} normal to the boundary, the normal component $\bb\cdot \nu$ is well defined by application of the Divergence Theorem to vector fields $\zeta \bb$, where $\zeta$ is a test function defined in the closure of $Q_L$. (See Section IX.2 of \cite{DautLionBook3} for a description of spaces of divergence-free $\mathrm{L}^2$ vector fields.) Omitting parameters in \eqref{functional} means that length has been non-dimensionalized in such a way that the intrinsic scale of the pattern is unity. Energy has been non-dimensionalized so that the energy per $(d-1)$-dimensional area is of order one.

\medskip

 The problem could equivalently be formulated in $Q_L\cup (-Q_L)$, where the constraint on $\infbnd$ could be formulated as a divergence equation in the distributional sense. This would be closer to the charge/field meaning of the pair $(u,\bb)$. However, by the symmetry of this problem, we consider only the upper cube.\footnote{From a candidate in $Q_L\cup (-Q_L)$ , one can define a symmetric candidate with no more energy by superposing the field $\bb/2$ and an appropriate reflection of $\bb/2$.} As explained in the introduction, we seek a result on mesoscopically uniform energy density. To obtain this, we will study the minimizer on smaller cubes $Q_l$, and the lower parts of their boundaries $\underline{Q_l}$.
When studying a global minimizer on a smaller scale, we need to take into account the influence of the whole domain. This takes the form of flux (Neumann) boundary conditions imposed on the upper parts of the boundary of $Q_L$: $\Gamma_L\defeq \partial Q_L \backslash \underline{Q_L}$ (or $\Gamma_l$, for $Q_l$). We thus use various types of boundary conditions which we list here. Implicitly, we already considered the family of candidates with free boundary condition on $\Gamma_L$:
\begin{align*}
\mathcal{A}(Q_L) \coloneqq \left \{ (u,\bb) \,\middle \vert\, u \in \mathrm{BV}(\infbnd,\pm 1), \bb \in \mathrm{L}^2(Q_L, \mathbb{R}^d), 
\begin{cases}
\Div \bb = 0 \text{ in } Q_L,\\
\, \,\bb \cdot \mathbf{\nu}= u 
\text{ on } \underline{Q_L}.
\end{cases}
\right  \}
\end{align*}
We also consider subclasses of $\mathcal{A}(Q_L) $ corresponding to
various flux boundary conditions on $\Gamma_L$ (again interpreted in the distributional sense). For $g \in \mathrm{L}^2(\Gamma_L)$, we will consider
\begin{align*}
 \mathcal A^{g}(Q_L) :=&
  \{(u,\bb)\in \mathcal A(Q_L): 
\bb \cdot \mathbf{\nu}= g \text{ on } \Gamma_L \}.
 \end{align*}
Of course, $\mathcal{A}^{g}(Q_L)$ is only non-empty if $\vert \int_{\Gamma_L} g \vert\leq L^{d-1}$, so that both the divergence-free condition and the boundary condition $\bb\cdot \nu =\pm 1$ on $\underline{Q}_L$ can be satisfied.
 A particular subclass which we will often study is the class of \textit{zero flux} candidates
\begin{align*}
 \mathcal A^{0}(Q_L) :=&
  \{(u,\bb)\in \mathcal A(Q_L): 
\bb \cdot \mathbf{\nu}= 0 \text{ on } \Gamma_L \}.
\end{align*}

\medskip
 
Finally, it is particularly convenient to work in a horizontally periodic setting, as it is invariant under horizontal translations. The charge $u$ will then live in the torus $\mathbb{T}_L^{d-1}$ (the $(d-1)$ dimensional cube of side length $L$, identifying the opposite faces in the usual sense) and the field $\bb$ in $\mathbb{T}_L^{d-1}\times (0,L)$. The boundary condition for $\bb$ on the top face will be free in this case; and we have to define the energy slightly differently:
\[
    E^{\per} (u,\bb, Q_L) \defeq  \int_{\mathbb{T}_L^{d-1}} \vert \nabla u\vert + \int_{\mathbb{T}_L^{d-1}\times (0,L)} \dfrac{1}{2} \vert \bb\vert^2.
\]
This notation suggests that we have identified $(u,\bb)$ on $\mathbb{T}_L^{d-1}\times (0,L)$ with horizontally periodic functions in $\R^{d-1}\times (0,L)$. In particular the divergence-free condition holds on the whole domain. However, we note that the energy takes into account (half of) the interface concentrated on $\partial(\underline{Q_L})$.
 
\medskip

 We define the {\it optimal energy densities} corresponding to various classes candidates as
\begin{eqnarray}
\sigma(Q_L) &\coloneqq& \inf_{(u,\bb) \in \mathcal{A}(Q_L)} \Frac{E(u,\bb,Q_L)}{L^{d-1}} ,\nonumber \\
\sigma^0(Q_L) &\coloneqq& \inf_{(u,\bb) \in \mathcal{A}^0(Q_L)} \Frac{E(u,\bb,Q_L)}{L^{d-1}},\nonumber \\
\sigma^{\per}(Q_L) &\coloneqq& \inf_{(u,\bb)\in \mathcal{A}^{\per}(Q_L)} \Frac{E^{\per}(u,\bb, Q_L)}{L^{d-1}} \label{eq:def-sigper}.
\end{eqnarray}
We note that, using the direct method of the calculus of variations, one can easily show that a minimizer of the functional \eqref{functional} exists in each of the classes
$\mathcal A(Q_L)$, $\mathcal A^{\per}(Q_L)$, $\mathcal A^{0}(Q_L)$ and $\mathcal A^{g}(Q_L)$ (for the latter, only if it is non empty).

\medskip

We use the short-hand notation $\lesssim$, $\gtrsim$ for $\leq C$ and $\geq C$ with a constant $C\in(0,+\infty)$ depending only on the dimension $d$, $\sim$ stands for $\lesssim$ and $\gtrsim$ at the same time. Furthermore, a hypothesis of the form $H \ll 1$  for some quantity $H$, means that there exists a constant $C\in(0,+\infty)$ (still depending only on $d$) such that $H\leq C^{-1}$.

\section{Statement of the main results}\label{state-main}
The main results of this paper are the following two theorems:
\begin{theorem}[Uniform distribution of energy]\label{final_convergence}
  There exists a constant $\sigma^*\in (0,+\infty)$, depending only on $d$, 
such that for $L\gg 1$,
\begin{equation}\label{eq:densityconv}
\max\left \{\left\vert \sigma(Q_L)-\sigma^*\right\vert, \left\vert \sigma^0(Q_L)-\sigma^*\right\vert , \left\vert \sigma^{per}(Q_L)-\sigma^*\right\vert \right \} \lesssim \dfrac{1}{L^{1/2}}.
\end{equation}
Furthermore, if $(u,\bb)$ is a minimizer in $\Aper$ and $L\geq l \gg 1$, then there holds
\begin{equation}\label{eq:unif_dis}
\left \vert \dfrac{E(u,\bb,Q_l)}{l^{d-1}} - \sigma^* \right \vert\lesssim  \dfrac{1}{l^{1/2}}.
\end{equation}
\end{theorem}
We have no reason to believe that the exponent $1/2$ in (\ref{eq:densityconv}) and (\ref{eq:unif_dis}) is optimal in any sense. However, it comes up naturally through Lemma \ref{pointwise_estimate}.
 It should be compared to the (better) exponent 1 in the case of \cite{ACO09}, which however
 can be improved by using the first variation (see Proposition 6.1 in that paper). We do not explore this direction in the present paper.
A scaling argument similar to one used in \cite[Theorem~1.2]{ACO09} yields:
\begin{theorem}[Equipartition of the energy]\label{equipartition}
For $(u,\bb)$ as above, if $L\geq l\gg 1$, then there hold
\[
    \left \vert \frac{1}{l^{d-1}}\int_{\underline{Q_l}} \vert \nabla u\vert -\dfrac{\sigma^*}{2}\right \vert \lesssim \dfrac{1}{l^{1/4}}
\quad \text{and} \quad
    \left \vert \frac{1}{l^{d-1}}\int_{Q_l} \dfrac{1}{2} \vert \bb\vert^2 -\dfrac{\sigma^*}{2}\right \vert \lesssim  \dfrac{1}{l^{1/4}}.
\]
\end{theorem}
%
%
%
\section{Uniform energy bound}\label{section_first_bound}
As in \cite{ACO09}, the first main step in establishing Theorems \ref{final_convergence} and \ref{equipartition} is a
uniform bound on the local energy density:

\begin{theorem}[Uniform bound on the energy density] \label{uniform_bound}
Let $(u,\bb)$ be a minimizer of \eqref{functional} in the class $\Aper$ with $L\gg 1$. Then given $l$ with $L\geq l \gg 1$, there holds
 \begin{equation}\nonumber
 \frac{E(u,\bb, Q_l)}{l^{d-1}}   \lesssim 1.
\end{equation}
\end{theorem}

(Since we formulate Theorem \ref{uniform_bound}
in terms of the periodic problem, there
is no loss of generality in considering centered cubes $Q_l$.)
Clearly, the task at hand is to pass the global energy estimate 
down to a local one, which will be done iteratively. 
The global energy estimate is a consequence of the
following easy lemma. It collects all the obvious relations,
including the natural monotonicities which follow from cutting and pasting,
and some easy estimates on the various global energy densities.

\begin{lemma}[Basic inequalities]\label{basic} There exists a constant $C\in(0,+\infty)$, depending only on $d$, such that for $L\gg 1$:
\begin{enumerate}[(i)]
\item $\sigma(Q_L)\;\;\;\leq \sigma^{\per}(Q_L)$ and $\sigma(Q_L) \leq\sigma^0(Q_L)$,\label{sigma_basic_comparison}
\item \label{basic free kl} $\sigma(Q_{L}) \;\;\;\leq \sigma(Q_{kL})$ for each positive integer $k$,
\item \label{basic 0 kl}$\sigma^0(Q_{kL}) \leq \sigma^0(Q_L)$ for each positive integer $k$,
\item $\sigma^0(Q_L) \;\;\leq C$, \label{sigma_0_basic_bound}
\item \label{basic lower bound} $\sigma(Q_L) \;\;\;\;\geq 1/C$,
\item \label{basic per free} $\sigma^{\per}(Q_{L}) \leq \sigma(Q_{L}) + C/L$.
\end{enumerate}
\end{lemma}

As in \cite{ACO09}, the main challenge in establishing Theorem \ref{uniform_bound}
consists in controlling the long-range interaction via the field.
This relies on screening in the sense of electrostatics, 
i.~e.~the reduction of the size of the
field $\bb$ through a rearrangement of the charges $u$. More precisely,
the issue is how effective this screening is in the presence
of some flux boundary data $g$ on $\Gamma_L$. 
If the charges $u$ were not constrained at all,
we would arrive at the following ``over-relaxed'' problem:
\begin{equation}\nonumber
E^g_0 (Q_L) \coloneqq \inf \left \{
\int_{\bulk} \Frac{1}{2} \vert \bb \vert^2
\middle \vert 
\begin{array}{l}
\Div \bb = 0 \text{  in  } \bulk,\\
\,\ \bb\cdot \nu = g  \text{  on  } \Gamma_L.
\end{array}
\right \}
\end{equation}
It is standard that a divergence-free field minimizing this energy is the gradient of a potential
$\bb = - \nabla v_0$ (see Section IX.3 of \cite{DautLionBook3}). Furthermore, because of the free boundary conditions on $Q_L$, the first variation of the field energy yields that $\bb$
is $\mathrm{L}^2$-orthogonal to all divergence-free vector fields $\tilde \bb$ (not necessarily vanishing on $Q_L$). 
By integration by parts, this implies the vanishing of the boundary integral of $v_0 \tilde \bb\cdot\nu$. Since any flux boundary data $g$ of vanishing boundary integral can be extended to a divergence-free field $\tilde \bb$, this implies that the trace of $v_0$ is orthogonal to all functions $g$ of vanishing boundary integral, and thus has to be constant. Therefore $v_0$ can be chosen as the solution to
\begin{equation}\label{eq:overrelaxed}
\begin{cases} -\Delta v_0 &=0 \quad \text{  in  } \bulk,\\
 -\nabla v_0 \cdot \nu &= g \quad \text{  on  } \Gamma_L,\\
 \quad\ \  v_0 &=0 \quad \text{  on  } \infbnd.
\end{cases}
\end{equation}

\medskip

Even in this ideal situation, screening is incomplete: While
reducing the horizontal components of the field near the surface 
$\underline{Q}_L$, 
it has little effect on the vertical component. 
The main insight is that, on large scales, 
the non-relaxed problem essentially has the same decay properties
as the over-relaxed one when passing from $Q_l$ to the smaller $Q_{\theta l}$; 
this is the content of the following lemma. We consider the quantity $F(\shift,l)$, corresponding to the {\it volume-averaged} energy of a fixed pair $(u,\bb)$, at scale $l$ after a vertical shift $\shift$:
\begin{equation}\label{eq:def-F}
F(\shift,l)\defeq \Frac{E(u,\bb-\beta \be_d,Q_l)}{l^d}.
\end{equation}

\begin{lemma}[One-step improvement]\label{itstep}
Let $(\bb,u)$ be a minimizer in $\Aper$. There exist constants $\delta \in (0,+\infty)$ and $\theta\in(0,1/2\,]$, depending only on $d$, such that the following holds. If $L\geq l\gg 1$ and 
\begin{equation}\label{eq:iteration-condition}
\shift\in [\,- 1/2, 1/2\,] \quad \text{is such that} \quad F(\shift, l) \leq \delta,
\end{equation}
then there exists a new shift $\tilde{\shift}$ such that 
\begin{equation}\label{eq:itstep}
\vert\shift-\tilde{\shift}\vert\lesssim F(\shift,l)^{1/2} \quad \text{and} \quad F(\tilde{\shift},\theta l) - \theta F(\shift, l)\lesssim \dfrac{1}{l}.
\end{equation}
\end{lemma}

In regularity theory such a result is known as a one-step
improvement lemma in a Campanato iteration. 
As usual in Campanato's characterization of H\"older spaces,
the (squared) volume average $F$ involves constant shifts, 
which in view of our comments after
\eqref{eq:overrelaxed}
reduces to the vertical component (and thus is parameterized by a
scalar $\shift$). As usual in this theory,
Lemma \ref{itstep} feeds into a Campanato iteration, of which we just retain
how the error term $Cl^{-1}$ in \eqref{eq:itstep} affects
the bound on small scales, which thanks to the {\it volume} average
in $F$ is still finer information than needed for the {\it area} average
in Theorem \ref{uniform_bound}.

\begin{lemma}[Campanato iteration]\label{iter}
Let $(\bb, u)$ be minimizing in $\Aper$.
There exists a constant $\delta \in(0,+\infty)$, depending only on $d$, such that the following holds:
If $F(0,L) \leq \delta$ and $L\geq l\gg 1$, then there holds
\begin{equation}\nonumber
F(0,l) \lesssim F(0,L)+\dfrac{1}{l}.
\end{equation}
\end{lemma}

We now explain the route towards Lemma \ref{itstep}. Note that \eqref{eq:itstep} could be strengthened to
$F(\tilde\shift,\theta l)-\theta^\alpha F(\shift,l)\leq C_\alpha l^{-1}$
for any $\alpha \in(0,2)$. In fact, the simpler $F(\tilde\shift,\theta l)\leq\theta^\alpha F(\shift,l)$ would be obvious on the level of the
over-relaxed problem.
The main work consists in appealing to local optimality for $(\bb,u)$
in order to lift this to the non-convex problem, at the expense
of the error term $Cl^{-1}$. Following \cite{ACO09}, this is
done via the convex ``relaxed'' problem, which in this paper plays an intermediate role
between non-relaxed and over-relaxed problem 
(and is just needed in this section):
\begin{equation}\label{eq:relaxed}
E_{rel}^g(\bulk) \coloneqq \inf \left \{
\int_{\bulk} \Frac{1}{2} \vert \bb \vert^2\ 
\middle \vert \begin{array}{ccc}
\Div \bb &=0 &\text{in  } \bulk,\\
 \bb \cdot \nu &\in[\,-1,1\,] &\text{on  } \infbnd,\\
\bb\cdot \nu &=g &\text{on  } \Gamma_L.
\end{array}
\right \}
\end{equation}
In Lemma \ref{shiftnonlin}, we show that indeed the over-relaxed problem is close
to the relaxed problem; and in Lemma \ref{nonconv}, we establish that the
relaxed problem is close to the original one in terms of energy.
\begin{lemma}[Closeness of over-relaxed to relaxed problem]\label{shiftnonlin}
Let $\bb$ be a solution to the relaxed problem  \eqref{eq:relaxed} and $\shift \in [\,-1/2,1/2\,]$ be fixed. 
Let $-\nabla v_0$ be the solution to the over-relaxed problem \eqref{eq:overrelaxed} with flux boundary data replaced by $(\bb - \shift \be_d) \cdot \nu$ on $\Gamma_L$. 
Then, supposing $\int_{\Gamma_L} ((\bb -\shift \be_d) \cdot \nu )^2 \ll L^{d-1} $, 
there holds
\begin{align*}
\int_{Q_L} \vert \bb - \shift \, \be_d + \nabla v_0\vert ^2 \lesssim \left (\int_{\Gamma_L} (( \bb - \shift \, \be_d )\cdot\nu ) ^2\right )^{\frac{d}{d-1}}.
\end{align*}
\end{lemma}

\begin{lemma}[Closeness of relaxed to original problem]\label{nonconv}
Given a minimizer $\overbar{\bb}$ for the relaxed problem \eqref{eq:relaxed},
 there exist $(\bb,u)\in \mathcal A^g(Q_L)$ such that
\begin{equation}\nonumber
      \int_{\underline{Q_L}} \vert \nabla u \vert  + \int_{Q_L} \Frac{1}{2} \vert \bb -\overbar{\bb} \vert ^2 \lesssim L^{d-1}.
  \end{equation}
\end{lemma}

While the proof of Lemma \ref{nonconv} is a straightforward post-processing,
which relies on elementary elliptic estimates in form of Lemma \ref{harm_building_block} below,
Lemma \ref{shiftnonlin} is more subtle. In fact, we first establish a version of
Lemma \ref{shiftnonlin} without the shift.

\begin{lemma}[Nonlinear estimate]\label{nonlinestimate}
Given $g$ with $\int_{\Gamma_L} g^2 \ll L^{d-1}$, if $\bb $ is a minimizer for the relaxed problem \eqref{eq:relaxed}, and $-\nabla v_0$ is the minimizer of the over-relaxed problem \eqref{eq:overrelaxed}, then there holds
\begin{align*}
\int_{Q_L} \Frac{1}{2} \vert \bb + \nabla v_0 \vert^2 
\lesssim \left (\int_{\Gamma_L} g^2 \right )^{\frac{d}{d-1}}.
\end{align*}
\end{lemma}

Here, as in Lemma \ref{shiftnonlin}, the crucial aspect of this non-linear estimate
is that the exponent $d/(d-1)$ appearing on the r.~h.~s. is 
(strictly) larger than one. Both Lemma \ref{nonlinestimate} and \ref{shiftnonlin} crucially rely on an 
obvious $\mathrm{L}^2$-orthogonality between the over-relaxed problem and the two
others, which we state for drama.

\begin{lemma}[Orthogonality]\label{orthogonality_lemma}
 Let $\bb$ be a divergence-free field and $v_0$ be the solution to the over-relaxed problem \eqref{eq:overrelaxed} in $Q_L$ with $g = \bb\cdot \nu$ on $\Gamma_L$.
  Then there holds
 \begin{equation}\nonumber
 \int_{Q_L} \Frac{1}{2} \vert \bb + \nabla v_0 \vert^2 = \int_{Q_L} \Frac{1}{2} \vert \bb \vert^2 -\int_{Q_L} \Frac{1}{2} \vert \nabla v_0 \vert^2.
\end{equation}
\end{lemma}

Modulo Lemma \ref{orthogonality_lemma}, we follow the approach of \cite{ACO09} to establish Lemma \ref{nonlinestimate}:
We appeal to a dual formulation of the convex relaxed problem (Lemma \ref{dual}),
which reduces Lemma \ref{nonlinestimate} to a trace estimate (Lemma \ref{interpolation}).

\begin{lemma}[Dual to the relaxed problem]\label{dual}
Given $g \in \mathrm{L}^2(\Gamma_L)$ with $\vert \int_{\Gamma_L} g\vert \leq L^{d-1}$ there holds
\begin{align}
E^g_{rel}(Q_L) =
 - \inf\left\{ \int_{Q_L} \Frac{1}{2} \vert \nabla v \vert^2  + \int_{\infbnd} \vert v \vert - \int_{\Gamma_L} vg \ \middle \vert \ v\in \mathrm{H}^1(Q_L) \right \}.\label{eq:dualtorel}
\end{align}
\end{lemma}

\begin{lemma}[Trace estimate]\label{interpolation}
Given a function $w\in \mathrm{H}^1(Q_L)$, then for any $\eps \in (0,1\,]$ there holds
\begin{align*}
\left (\int_{\underline{Q_L}} w^2 \right )^{\frac{1}{2}} \lesssim (\eps L)^{\frac{1}{2}} \left (\int_{Q_L} \vert \nabla w \vert^2 \right )^{\frac{1}{2}} + \dfrac{1}{(\eps L)^{\frac{d-1}{2}}}\int_{\underline{Q_L}} \vert w \vert.
\end{align*}
\end{lemma}

We conclude this section by stating two elliptic estimates that we need. We could not find a reference for the maximal regularity expressed in Lemma \ref{Neumann} and thus provide a proof. Lemma \ref{harm_building_block} is standard and we include its elementary proof for convenience.

\begin{lemma}[Normal flux estimate]\label{Neumann}
For any harmonic function $v$ on the cube $[0,L]^d$ which vanishes on its bottom
face $x_d=0$ we have
for the normal derivative $\partial_\nu v$ on the boundary $\partial(0,L)^d$:
\begin{equation}\nonumber
\int_{\{x_d=0\}}(\partial_\nu v)^2\lesssim \int_{\Gamma}(\partial_\nu v)^2,
\end{equation}
where $\Gamma:=(\partial(0,L)^d)\setminus \{x_d=0\}$.
\end{lemma}

\begin{lemma}[Harmonic building block] 
\label{harm_building_block}
 Let $g$ be such that $\int_{\partial Q_L} g =0$ and solve
 \begin{equation}\nonumber
  \begin{cases}
    \Delta v &= 0 \quad \text{in } Q_L, \\
  \nabla v \cdot \nu &= g \quad \text{on } \partial Q_L.
  \end{cases}
 \end{equation}
 Then for $p\in[\,2(d-1)/d,2\,]$ there holds 
 \begin{equation*}
  \int_{Q_L} |\nabla v|^2 \lesssim L^{d-(d-1)\frac{2}{p}} \left(\int_{\partial Q_L} \vert g \vert ^p\right)^{\frac{2}{p}}.
 \end{equation*}
\end{lemma}
\section{Pointwise estimates on the fields} \label{section_pointwise}

This technical section is devoted to pointwise bounds on the field $\bb$ coming
from a fixed energy minimizer $(u,\bb)$ in ${\mathcal A}^{per}(Q_L)$. Rather than the
field $\bb$, we are interested in its potential $v$, which in fact is also
horizontally periodic:

\begin{lemma}[Periodicity of the  potential $v$] \label{v_periodic}
There exists $v$ such that $\bb = - \nabla v$, $v$ is periodic in the horizontal directions and $v=0$ on $\{x_d=L\}$.
\end{lemma}

We think of $v$ and $\bb$ as fields on $\mathbb{R}^{d-1}\times(0,L)$ that
are horizontally periodic; we are also interested in the corresponding (non-periodic)
over-relaxed problem 
$E^{\bb\cdot\nu}_0(Q_l)$ on $Q_l$ with flux boundary data given by $b\cdot \nu$ on $\Gamma_l$, 
cf.~(\ref{eq:overrelaxed}), for some $l\leq L$.
The upcoming pointwise bounds are on both potentials $v$ and $v_0$, which share the same flux boundary data on $\Gamma_l$,
as well as on the charge density $-\partial_d v_0$ on $\underline{Q}_l$ coming from the over-relaxed potential:

\begin{lemma}[Pointwise bound on the potentials] \label{hoelder_v}
The potential $v$ and $v_0$ are uniformly H\"older-$1/2$ continuous, i.e.  letting $[\, \cdot \,]_{\mathrm{C}^{1/2}(Q)}$ stand for the H\"older-$1/2$ semi-norm on the cube $Q$, there holds
\begin{equation}\nonumber
[\, v\,]_{\mathrm{C}^{1/2}(Q_L)} \lesssim 1 \quad \text{and} \quad
[\, v_0\,]_{\mathrm{C}^{1/2}(Q_l)} \lesssim 1.
\end{equation}
\end{lemma}

\begin{lemma}[Pointwise bound on the over-relaxed charge density]\label{pointwise_estimate}
Let $Q_l$ and $v_0$ be as in 
the previous lemma. Then there holds
 \begin{equation} \nonumber
|\partial_d v_0 (x',0) |\lesssim \dist(x',\Gamma_l)^{-\frac{1}{2}} \quad \text{for all\, } (x',0)\in\underline{Q_l}.
\end{equation}
\end{lemma}

These two results will be crucially used in the constructions of Section \ref{section-last};
they just rely on the a priori bound of Theorem \ref{uniform_bound} and 
follow from elementary potential theory. To be more precise,
Lemma \ref{pointwise_estimate} is a straightforward consequence of the $v_0$-part of Lemma \ref{hoelder_v},
which in view of the relation between $v_0$ and $v$ easily follows
from the $v$-part of Lemma \ref{hoelder_v}. The $v$-part of Lemma \ref{hoelder_v}
is an easy consequence of the following two pointwise bounds on $\bb=-\nabla v$:

\begin{cor}[Pointwise bound on the field away from the surface]\label{firstdecay}
Given $(x',x_d)\in Q_L$ with $1\ll x_d \leq L$, there holds
\begin{equation*}
    \vert \bb (x',x_d) \vert \lesssim x_d\,^{-\frac{1}{2}}.
\end{equation*}
\end{cor}

\begin{lemma}[Pointwise bound on the field near the surface]\label{small_scales}
Given $(x',x_d) \in Q_L$ with $x_d\ll1$, there holds
\[
    \vert \bb(x',x_d)\vert \lesssim \ln\left (\dfrac{1}{x_d}\right ).
\]
\end{lemma}

The bound of Corollary  \ref{firstdecay}  is an immediate consequence of Theorem  \ref{uniform_bound} ,
whereas the bound of Lemma \ref{small_scales}, which we only use to extend the bound of
Corollary \ref{firstdecay} to the range $x_d\ll 1$, follows
by elementary potential theory from the boundedness of the charge density.
%
%
%
%
\section{Boundary conditions are negligible in the thermodynamic limit.}\label{section-last}
This section leads to Theorem \ref{final_convergence}
via a series of lemmas. 
 Throughout the section, $(u,\bb)$ denotes a minimizer in $\Aper$ with $L\gg 1$.  
The main strategy for the proof of Theorem \ref{final_convergence}
is to relate the local energy $E(u,\bb,Q_l)$ (always in the regime $l\gg 1$)
to the {\it sum} of $\sigma^{0}(Q_l)$ (or $\sigma^{per}(Q_l)$) 
and the energy $\int_{Q_l} \frac{1}{2} |\nabla v_0|^2$ of the solution $v_0$ to the over-relaxed 
problem \eqref{eq:overrelaxed} on $Q_l$ with flux boundary data $\bb\cdot \nu$ on $\Gamma_l$. On the level
of the lower bound, this is done in Lemma \ref{sigma_null_sigma}; for the upper bound, it is Lemma \ref{sigma_sigma_null}:

\begin{lemma}[Precise lower bound on the local energy]\label{sigma_null_sigma}
There exists a constant $C$, depending only on $d$, such that if $L\geq l\gg 1$, then there holds
\[
   \dfrac{1}{l^{d-1}}\left (E(u,\bb,Q_l)- \int_{Q_l} \dfrac{1}{2} \vert \nabla v_0\vert^2 \right) \geq \sigma^{0}(Q_l) -\dfrac{C}{l^{1/2}}.
\]
\end{lemma}
\begin{lemma}[Precise upper bound on the local energy]\label{sigma_sigma_null}
There exists a constant $C$, depending only on $d$, such that if $L\geq l\gg 1$, then there holds
\begin{equation}\label{eq:E_le_sigper}
  \dfrac{1}{l^{d-1}} \left (E(u,\bb, Q_l) - \int_{Q_l}\frac{1}{2}|\nabla v_0|^2\right) \leq  \sigma^{per}(Q_l) +  \frac{C}{l^{1/2}}.
\end{equation}
\end{lemma}

Of course, both lemmas are established by constructions:
In Lemma \ref{sigma_null_sigma} we construct a candidate for ${\mathcal A}^0(Q_l)$ based on $(u,\bb+\nabla v_0)$,
which by definition of $-\nabla v_0$ has vanishing flux boundary data,
while Lemma \ref{sigma_sigma_null} uses a minimizer $(u^{per},\bb^{per})$ in ${\mathcal A}^{per}(Q_l)$ 
to construct a competitor for $(u,\bb)$ in $Q_l$ based on $(u^{per},\bb^{per}+\nabla v_0^{per}-\nabla v_0)$,
which 
 has the right boundary conditions if $v_0^{per}$ is the solution to the over-relaxed problem \eqref{eq:overrelaxed} with flux boundary data $g\defeq \bb^{per}\cdot \nu$ on $\Gamma_l$.
Loosely speaking, both lemmas express an orthogonality between the 
micro-structured part and the over-relaxed part, which is again based on Lemma \ref{orthogonality_lemma}. 
Both lemmas express this relationship on the level of a relative error decaying as
$l^{-1/2}$. An easy consequence of Lemma \ref{sigma_null_sigma} for $l=L$ and of Lemma \ref{basic} is the first
part of Theorem \ref{final_convergence} in the form of

\begin{cor}[Boundary conditions do not affect the global energy density] \label{existence_of_limit}
 The limit
 \begin{equation*}
  \sigma^* = \lim_{L\to\infty} \sigma (Q_L) = \lim_{L\to\infty} \sigma^0 (Q_L) = \lim_{L\to\infty} \sigma^{per} (Q_L)
 \end{equation*}
exists, with the convergence rate
\begin{equation*}
 \max\left \{|\sigma(Q_L) - \sigma^* |,|\sigma^0(Q_L) - \sigma^* |,|\sigma^{per}(Q_L) - \sigma^* |\right\}\lesssim \dfrac{1}{L^{1/2}}.
\end{equation*}
\end{cor}

The second part of Theorem \ref{final_convergence} follows from Lemmas \ref{sigma_null_sigma} and \ref{sigma_sigma_null} once we establish
that the energy of the over-relaxed problem is negligible:

\begin{lemma}[Over-relaxed solutions are energetically negligible]\label{control-v_0}
Given $L\gg 1$ and $\theta \ll 1$ such that $\theta^{-1} \in 2\N+1$, 
if $l\gg 1$ is of the form $l= \theta^N L$ for some positive integer $N$ and $v_0$ is the solution to the over-relaxed problem \eqref{eq:overrelaxed} induced by $\bb$ on $Q_l$, then there holds
\[
    \Frac{1}{l^{d-1}}\int_{Q_l} \Frac{1}{2}\vert \nabla v_0\vert^2 \lesssim \Frac{1}{l^{1/2}}.
\]
\end{lemma}

While not very technical, Lemma \ref{control-v_0} is subtle: The combination of Lemmas \ref{sigma_null_sigma} and \ref{sigma_sigma_null}
(in conjunction with Corollary \ref{existence_of_limit}) allows us to relate the over-relaxed solution of 
a big cube to those of smaller cubes that partition the big one. This drives yet
another Campanato iteration based on the (boundary) regularity of the over-relaxed problem,
akin to Lemmas \ref{itstep} and \ref{iter}.

\medskip

The constructions of Lemmas \ref{sigma_null_sigma} and \ref{sigma_sigma_null} are technical: In the case of Lemma \ref{sigma_null_sigma}, the pair
$(u,\bb+\nabla v_0)$ is not admissible in $\mathcal{A}^0(Q_l)$ due to the additional charge $-\partial_dv_0$
on $\underline{Q}_l$. Thanks to Lemma \ref{pointwise_estimate}, this extra charge is small outside of a boundary layer. In order to achieve the related scaling $l^{-1/2}$ of the
relative error, we adjust for this small amount of extra charge by a smooth
deformation of the set described by the characteristic function $u$. 
This type of deformation of $u\in\{-1,1\}$ into another characteristic function $\tilde u$ 
in order to change the volume (fraction) while controlling the surface area is folklore. We did not find in the literature a suitable statement for the quantitative modification of a set of finite perimeter, though one might
try and start from the perturbation defined in \cite{EspoFusc2011remark}. We state (and prove) what we need for the convenience of the reader:

\newcommand{\cube}{{\underline{Q}}}
\begin{lemma}[Smooth deformation of sets to adjust their volume] \label{lemma_flow}
Let $\lambda\geq 1$ and $\cube=(0,\lambda)^{d-1}$ be a $(d-1)$-dimensional cube. Consider a function $u=\{1,-1\}$ on $\cube$ with
\begin{align} 
    \left \vert \dashint_\cube u \right \vert &\leq \frac{1}{2}, \label{control_avg} \\
    \int_\cube |\nabla u| &\leq \lambda^{d-2} \Lambda, \label{control_perim}
\end{align}
for some $\Lambda <+\infty$. Then there exists a constant $m_0\in(0,+\infty)$, depending only on $\Lambda$ and $d$, such that for any $m\in [\,-m_0,m_0\,]$ there exists $\tilde u=\{-1,1\}$ with $(u-\tilde u)$ compactly supported in $\cube$ and
\begin{align*}
    \dashint_\cube\tilde  u = \dashint_\cube u +m,& \qquad
    \dashint_\cube |\tilde u -  u| \lesssim_\Lambda |m|, \\
    \int_\cube |\nabla \tilde u| - \int_\cube |\nabla u| &\lesssim_\Lambda \lambda^{d-2}|m|, 
\end{align*}
where the implicit constants in the last two inequalities depend only on $\Lambda$ and $d$.
\end{lemma}

Our application of Lemma \ref{lemma_flow} relies on the fact that we have ``enough room'' to 
perturb the local volume fraction. 
This is a consequence of Theorem \ref{uniform_bound}, which ensures that
the charge distribution $u$ is approximately neutral on large scales:

\begin{lemma}[Neutrality on large-scale averages]\label{control_u_avg}
 If $(u,\bb)$ is a minimizer in $\mathcal{A}^{per}(Q_l)$ and $l\geq \lambda \gg 1$, then there holds
 \begin{equation*}
  \left|\dashint_{\underline Q_\lambda} u \right| \leq \frac{1}{2}.
 \end{equation*}
\end{lemma}

\section{Proofs}\label{section_Proofs}
In this section, we provide the proofs of the theorems, lemmas, and corollaries in the order they were stated.
\begin{proof}[Proof of Theorem \ref{final_convergence}]
The convergence of $\sigma(Q_L)$, $\sigma^{per}(Q_L)$, and $\sigma^0(Q_L)$ to $\sigma^*$ as $L\uparrow +\infty$, is established in Corollary \ref{existence_of_limit}.
We turn to (\ref{eq:unif_dis}). We start by considering a cube $Q_l$ contained in $Q_L$ with $l=\theta^n L$ for some integer $n$ and some $\theta\sim 1$ with $\theta^{-1} \in 2\mathbb{N}+1$ such that Lemma \ref{control-v_0} holds. By Lemma \ref{sigma_null_sigma}, Lemma \ref{sigma_sigma_null}, and (\ref{eq:densityconv}) at scale $l$, there holds
\[
  \left \vert \dfrac{E(u,\bb, Q_l)}{l^{d-1}} - \dfrac{1}{l^{d-1}}\int_{Q_l}\dfrac{1}{2}|\nabla v_0|^2 - \sigma^*\right \vert \lesssim \dfrac{1}{l^{1/2}}.
\]
By Lemma \ref{control-v_0}  the term involving $\nabla v_0$ is controlled by $C l^{-1/2}$ and we may conclude for such cubes.

\medskip

So far, we have proved (\ref{eq:unif_dis})
for boxes of lateral size $l$ of the form $\theta^n L$, 
$n\in\mathbb{N}$,
where $\theta^{-1}\in2\mathbb{N}+1$ with $\theta\sim 1$ coming from Lemma \ref{control-v_0}. 
We notice that Lemma \ref{sigma_null_sigma} and Corollary \ref{existence_of_limit} imply
\begin{align}\label{as01}
E(u,\bb,Q)-\sigma^*|\underline{Q}|\gtrsim -l^{d-3/2}
\end{align}
for all boxes $Q$ of size $l\leq L$ (the statement is trivial for $l\ll 1$).
We note that the l.~h.~s.~$E(u,\bb,C)-\sigma^*|\underline{C}|$ is super-additive
in sets $C$ touching the bottom. Since
any $l\in[\,0,L)$ can be written as linear combination of $\{2^{-n}L\}_{n\in\mathbb{N}}$
with coefficients in $\{0,1\}$, and since $d-3/2>0$ we obtain that (\ref{as01})
holds for any non-square boxes $C$ of lateral size smaller than $l$, too.
For a given box $Q$, there exists a box $Q'$ of side-length of the form $\theta^nL$ 
and $d$ non-square boxes $C$ such that $\underline{Q'}$ is the (disjoint) union of
$\underline{Q}$ and the $\underline{C}$'s. Hence by the above-mentioned super-additivity,
and by (\ref{as01}), we may lift the opposite estimate
\begin{align*}
E(u,\bb,Q)-\sigma^*|\underline{Q}|\lesssim l^{d-3/2}
\end{align*}
from $Q'$ to $Q$.
\end{proof}

\begin{proof}[Proof of Theorem \ref{equipartition}] 
We closely follow the proof of \cite[Theorem~1.2]{ACO09}. Let $(u,\bb)$ be a minimizer in $\Aper$, by Theorem \ref{final_convergence} there exists $C\in(0,+\infty)$ such that if $L\geq l \gg 1$, then there holds
\begin{equation}\label{eq:f1}
    \dfrac{E(u,\bb,Q_l)}{l^{d-1}}\leq \sigma^* + \dfrac{C}{l^{1/2}}.
  \end{equation}
  On the other hand, for $\lambda\in [\, 1/2,3/2\,]$ the rescaled pair $(u^\lambda,\bb^\lambda)$ defined by
\[
    u^\lambda(\lambda x ) = u(x) \quad\text{ and }\quad \bb^\lambda(\lambda x) = \bb(x)
\]
is a candidate in $\mathcal{A}^{per}(Q_{\lambda L})$. In particular, the restriction of $(u^\lambda,\bb^\lambda)$ to $Q_{\lambda l}$ is a candidate in $\mathcal{A}(Q_{\lambda l}).$ Using this and \eqref{eq:densityconv} in Theorem~\ref{final_convergence} at scale $\lambda l$, we get
\begin{equation}\label{eq:f2}
    f(\lambda) \defeq \dfrac{E(u^\lambda,\bb^\lambda, Q_{\lambda l})}{(\lambda l)^{d-1}}\geq \sigma(Q_{\lambda l})
\geq \sigma^* - \dfrac{C}{(\lambda l)^{1/2}}.
\end{equation}
In this notation, the combination of (\ref{eq:f1}) and (\ref{eq:f2}) 
yields 
\begin{align}\label{co01}
f(\lambda)-f(1)\gtrsim-\frac{1}{l^{1/2}}\quad\mbox{and}\quad f(1)\lesssim 1.
\end{align}
It follows from a change of variables that $f$ is of the form 
\begin{align}\label{co02}
f(\lambda)=\frac{1}{\lambda}a+\lambda b\quad\mbox{with}
\;a:=\dfrac{1}{l^{d-1}} \int_{\underline{Q_l}}|\nabla u|,
\;b:=\dfrac{1}{l^{d-1}}\int_{Q_l}\frac{1}{2}|\bb|^2.
\end{align}
Since $f''(\lambda)=2a/\lambda^3\le 16 f(1)$, it follows from the second item in (\ref{co01})
that $|f''(\lambda)|\lesssim 1$, so that by Taylor's formula for $t\in [\,-1/2,1/2\,]$,
\begin{align*}
\Big|f'(1)- \dfrac{1}{t}\big(f(1+ t)-f(1)\big)\Big|
\lesssim |t|.
\end{align*}
Thus, for $t= \pm 1/l^{1/4}$ and by the first item in (\ref{co01}) this implies $|f'(1)|\lesssim1/ l^{1/4}$, which because of $f'(1)=b-a$, see (\ref{co02}), yields the claim of the theorem.
\end{proof}

\begin{proof}[Proof of Theorem \ref{uniform_bound}]
    Let $(\bb, u)$ be minimizing in $\Aper$. Provided $L\gg 1$, by statements (\ref{basic per free}), (\ref{sigma_basic_comparison}) and (\ref{sigma_0_basic_bound}) of  Lemma~\ref{basic}, and in view of definition (\ref{eq:def-sigper}), we have
\[
\dfrac{E(u, \bb, Q_L)}{L^{d-1}} \lesssim 1.
\]
 By definition (\ref{eq:def-F}) again provided  $L\gg 1$, we thus have $F(0,L)\leq C/L \leq \delta$, where $\delta$ is as in Lemma~\ref{iter}. Applying this Lemma for $L\geq l\gg 1$, we get
    \begin{equation*}
        E(u,\bb, Q_l) \lesssim \dfrac{l^{d}}{L^{d}}E(u,\bb,Q_L)+ l^{d-1} =l^{d-1}\left( \dfrac{l}{L}\dfrac{E(u,\bb,Q_L)}{L^{d-1}}+1 \right ) \lesssim l^{d-1},
    \end{equation*}
as desired.
\end{proof}

\begin{proof}[Proof of Lemma \ref{basic}]
To prove \eqref{sigma_basic_comparison} notice that a minimizer $(u,\bb)$ in $\mathcal{A}^{per}(Q_L)$, when thought of as peridically extended and then restricted to $Q_L$, is also a candidate in $\mathcal{A}(Q_L)$, and that $E^{\per}(u,\bb,Q_L) \geq E(u,\bb, Q_L)$, leading to $\sigma^{per}(Q_L) \geq \sigma(Q_L)$. Similarly, we have $\sigma^0(Q_L)\geq\sigma(Q_L)$.

\medskip

To prove \eqref{basic free kl}, let $k$ be a positive integer and consider a minimizer $(u,\bb)$ in $\mathcal{A}(Q_{kL})$. We decompose $Q_{kL}$ in $k^{d-1}$ cubes $\{Q^i_L\}$ of side length $L$ touching the lower boundary, 
and a large box above. 
Since by definition of $\sigma$, $E(u,\bb, Q_L^i) \geq L^{d-1} \sigma(Q_L)$, we obtain
\begin{equation*}
\sigma(Q_{kL}) = \dfrac{E(u,\bb,Q_{kL})}{(kL)^{d-1}}\geq \dfrac{1}{k^{d-1}} \sum_{i=1}^{k^{d-1}} \dfrac{E(u,\bb, Q_L^i)}{L^{d-1}} \geq \sigma(Q_L).
\end{equation*}

\medskip

Statement~\eqref{basic 0 kl} comes from concatenating $k^{d-1}$ translated copies of a minimizer in $\mathcal{A}^0(Q_L)$ and finally extending $\bb$ by $0$ on $x_d\geq L$. There is no added interface if each copy is the even reflection of its neighbour across the cubic face they share. We refer to Sections 3.6 to 3.8 in \cite{AmFuPa_book} for the continuity and trace properties that make it possible to cut and paste BV functions in neighbourhing cubes.

\medskip

To prove \eqref{sigma_0_basic_bound} notice that by \eqref{basic 0 kl} it suffices to show that $\sigma^0(Q_L) \lesssim 1$ for $L\in[\,1,2\,]$. This is seen by constructing $(u,\bb)\in \mathcal{A}(Q_L)$. Indeed, let $u(x_1,\dots,x_{d-1}) \defeq \sign x_1$ and
\[
   \bb (x_1,\dots,x_{d-1}, x_d) =\begin{cases} \dfrac{x_1}{\sqrt{x_1^2+x_d^2}} \be_d - \dfrac{x_d}{\sqrt{x_1^2+x_d^2}} \be_1, \quad &\text{if } x_1^2+x_d^2\leq L^2,\\
     \qquad 0 &\text{otherwise}.
   \end{cases}
   \] 
We notice that inside $\{x_1^2+x_d^2 \leq L^2\}$, this two-dimensional vector field is the rotated gradient of the ''stream function'' $\psi(x_1,x_d) = \sqrt{x_1^2 +x_d^2}-L$, and thus divergence-free and tangential to $\{x_1^2+x_d^2 = L^2\}$. Hence its trivial extension is weakly divergence-free. It vanishes on $\Gamma_L$, and satisfies $\bb\cdot \be_d = \sign x_1= u$ on $\underline{Q_L}$, so that the boundary conditions are also satisfied.

\medskip

We now turn to (\ref{basic lower bound}). As propagated in \cite{CintiOtto}, scaling-wise
optimal lower bounds in pattern-forming variational problems typically rely on 
interpolation inequalities that capture the leading-order competition between the
energy contributions. Here, the interpolation estimate involves the $BV$ norm
and the $\dot H^{-\frac{1}{2}}$ norm -- we now give an elementary proof.
 Choosing $l$ to be sufficiently large but of order one, it is enough to
show for $(u,\bb)\in{\mathcal A}(Q_L)$ that
\begin{align*}
L^{d-1}\lesssim l E(u,\bb,Q_L)+\frac{L^{d-1}}{l^2}.
\end{align*}
By Young's inequality, for this it suffices to  establish for $l\ll L$
\begin{align}\label{ap05}
|\underline{Q_{L-2l}}|\lesssim l\int_{\underline{Q_L}}|\nabla u|
+\Big(\frac{L^{d-1}}{l}\int_{Q_L}|\bb|^2\Big)^\frac{1}{2}.
\end{align}
To this purpose, we fix a cut-off function $\eta\in[\,0,1\,]$ with
\begin{align}\label{ap01}
\eta=\left\{\begin{array}{ccc}
1&\mbox{on}&{\textstyle(-\frac{L-2l}{2},\frac{L-2l}{2})^{d-1}\times\{0\}},\vspace{.1cm}\\
0&\mbox{out of}&{\textstyle(-\frac{L-l}{2},\frac{L-l}{2})^{d-1}\times[\,0,l),}
\end{array}\right\}
\quad\mbox{while}\quad
|\nabla\eta|\lesssim\frac{1}{l}.
\end{align}
By convolution at scale $l$, we construct $\tilde u\in[\,-1,1\,]$ such that
\begin{align}\label{ap02}
\int_{\underline{Q_{L-l}}}|\tilde u-u|\lesssim l\int_{\underline{Q_L}}|\nabla u|
\quad\mbox{while}\quad|\nabla\tilde u|\lesssim\frac{1}{l}.
\end{align}
In particular, we obtain from this and the support condition in (\ref{ap01})
\begin{align}\label{ap04}
\int_{\mathbb{R}^{d-1}}\vert \eta u(u-\tilde u)\vert \lesssim l\int_{\underline{Q_L}}|\nabla u|.
\end{align}
We now test (\ref{eq:constraints}) with $\zeta=\eta\tilde u$, where we think
of $\tilde u$ as being extended in a constant way to $x_d>0$. This yields
\begin{align*}
\int_{\mathbb{R}^{d-1}}\eta \tilde u u=-\int_{\mathbb{R}^d}\nabla\zeta\cdot \bb,
\end{align*}
which in view of $|\nabla\zeta|\lesssim 1/l$ (see the last items in (\ref{ap01}) and (\ref{ap02})) and the H\"older inequality,
implies
\begin{align}\label{ap03}
\int_{\mathbb{R}^{d-1}}\eta \tilde u u\lesssim\Big(\frac{L^{d-1}}{l}\int_{Q_L}|\bb|^2\Big)^\frac{1}{2}.
\end{align}
Because of the non-convex constraint in form of $u^2=1$, the sum of (\ref{ap04}) and
(\ref{ap03}) yields
\begin{align*}
\int_{\mathbb{R}^{d-1}}\eta\lesssim l\int_{\underline{Q_L}}|\nabla u|
+\Big(\frac{L^{d-1}}{l}\int_{Q_L}|\bb|^2\Big)^\frac{1}{2},
\end{align*}
which by (\ref{ap01}) turns into (\ref{ap05}).

\medskip

To prove (\ref{basic per free}), we let $(u,\bb)$ be an optimal configuration in $\mathcal{A}(Q_L)$. 
On one of the $2^{d-1}$ (horizontal) quadrants of $Q_L$, the energy per area does not exceed the original one, w.l.o.g. we may assume that this is the case
for $(0,L/2)^{d-1}\times(0,L)$. 
We extend the restriction of $(u,\bb)$ to this quadrant by reflections to obtain a candidate $(\bar u,\bar \bb)\in\mathcal{A}^{per}(Q_L)$. We
do this iteratively in the horizontal directions $i=1,\dots,d-1$:
$\bb\cdot \be_i$ is reflected evenly across $x_i=0$ in order to avoid a jump, $u$ and 
all the other components of $\bb$ are reflected oddly.
 This way, we obtain a new configuration $(\bar u,\bar \bb)$ on $Q_L$,
 the energy of which exceeds the one
of $(u,\bb)$ by at most the additional interfacial energy $2(d-1)L^{d-2}$. (We refer again to Sections 3.6 to 3.8 in \cite{AmFuPa_book} for the results that allow us to cut and paste BV functions.) Extending $(\bar u,\bar \bb)$ periodically in the horizontal directions adds interfacial energy by at most the same amount.
\end{proof}
%
\begin{proof}[Proof of Lemma \ref{itstep}] We consider a minimizer $(\bb,u)$ in $\mathcal{A}^{per}(Q_L)$. We suppose that for some $\delta\in(0,+\infty)$ to be chosen later, there holds $F(\shift, l) \leq \delta $,   for some $l$ and $\shift$ with $L\geq l \gg 1$ and $\shift \in[\,-1/2,1/2\,]$. 
We choose a smaller scale $\rho$ such that
\begin{align*}
\rho \defeq \mathrm{argmin} \left \{ \int_{\Gamma_r} \vert \bb-\shift \be_d \vert ^2 \ \middle \vert\  r \in \big [\,\textstyle{\frac{3}{4}}l
, l\,  \big]\right \}.
\end{align*}
By Fubini's theorem, this choice of $\rho$ implies that
\begin{equation}\label{eq:rhofluxestim}
	\dfrac{1}{\rho^{d-1}} \int_{\Gamma_\rho} \vert \bb-\shift \be_d \vert ^2 \lesssim \dfrac{1}{l^d}\int_{Q_l}\vert \bb-\shift \be_d \vert ^2 \, \overset{(\ref{eq:def-F})}{=}\,  F(\shift,l).
\end{equation}
Let $v_0$ be the solution of the over-relaxed problem \eqref{eq:overrelaxed} in $Q_{\rho}$ with flux boundary data $(\bb-\shift \be_d)\cdot \nu$ on $\Gamma_\rho$.
We note that the Dirichlet energy of $-\nabla v_0$ is less than the squared $\mathrm{L}^2$-norm of $\bb-\shift \be_d$, as $-\nabla v_0$ is obtained by relaxation.
 Extend $v_0$ harmonically to negative values of $x_d$ by odd reflection (this is possible because $v_0=0$ on $\underline{Q_{\rho}}$). We now shift the problem by $(\tilde{\shift}-\shift)\be_d$, where 
\begin{equation}
  \label{eq:tildeshift}
  \tilde{\shift}\defeq \shift  + \partial_d v_0(0).
\end{equation}
By the fact that $\nabla v_0(0)= \partial_d v_0 (0)\be_d$ and sub-harmonicity of $\vert \nabla v_0\vert^2$, we have
\begin{equation}\label{eq:shiftdiffestim}
( \tilde{\shift}-\shift )^2 \overset{\eqref{eq:tildeshift}}{=} |\nabla v_0(0)|^2\lesssim \dfrac{1}{\rho^d} \int_{Q_{\rho}} \vert \nabla v_0 \vert^2 \leq \dfrac{1}{\rho^d} \int_{Q_{\rho}} \vert \bb-\shift \be_d \vert^2 \overset{\eqref{eq:def-F}}{\lesssim} F(\shift, l).
\end{equation}
For $\theta\in (0,1/2\,]$ to be chosen later, by the triangle inequality we obtain
\begin{align}
\int_{\underline{Q_{\theta l}}} \vert \nabla u \vert + \int_{Q_{\theta l}} \Frac{1}{2}\vert \bb-\tilde{\shift}\be_d \vert^2 \hspace{-2cm}&\nonumber
\\
&\leq  \int_{\underline{Q_{\theta l}}} \vert \nabla u \vert + \int_{Q_{\theta l}} \vert \bb-\shift \be_d +\nabla v_0\vert^2  + \int_{Q_{\theta l}} \vert \nabla v_0-(\tilde{\shift}-\shift) \be_d \vert^2. \label{decomp}
\end{align}

\medskip

We will treat the sum of the two first terms and the third term separately. For the third term, the mean value property (Theorem~2.1 in \cite{GilbargTrudinger}) yields 
\begin{equation}\nonumber
 \dfrac{1}{(\theta l)^{d}}\int_{Q_{\theta l}} \vert \nabla v_0-(\tilde{\shift}-\shift) \be_d \vert^2 
\overset{\eqref{eq:tildeshift}}{=} \dfrac{1}{(\theta l)^{d}}\int_{Q_{\theta l}} \vert \nabla v_0- \nabla v_0(0)\vert^2
\lesssim (\theta l)^2 \sup_{Q_{\theta l}} \vert \nabla^2 v_0 \vert ^2.
\end{equation}
By inner regularity of harmonic functions (for instance Theorem~2.10 in \cite{GilbargTrudinger}, followed by the mean value property),
recalling that $\theta \in (0,1/2\,]$ we thus get
\begin{equation}\label{estim1}
\dfrac{1}{(\theta l)^{d}}\int_{Q_{\theta l}} \vert \nabla v_0-(\tilde{\shift}-\shift) \be_d \vert^2 
\lesssim (\theta l)^2\dfrac{1}{ \rho^{d+2}} \int_{Q_{\rho}} \vert \nabla v_0 \vert ^2
 \overset{\eqref{eq:shiftdiffestim}}{\lesssim} \theta^2 F(\shift,l).
\end{equation}

\medskip

We now compare the first two terms on the right hand side of inequality \eqref{decomp} with the energy of the relaxed problem, using Lemma \ref{nonconv}:
Define $\overbar{\bb}$ as the solution to the relaxed problem \eqref{eq:relaxed} on $Q_\rho$ with flux boundary values $\bb\cdot \nu$ across $\Gamma_{\rho}$. By Lemma \ref{nonconv}, one can find a candidate $(\tilde{u},\tilde{\bb})$ in $\mathcal{A}^{\bb\cdot \nu}(Q_{\rho})$  such that
\begin{equation} \nonumber
  \int_{\underline{Q_{\rho}}} \vert \nabla \tilde{u}\vert + \int_{Q_{\rho}} \Frac{1}{2}\vert \tilde{\bb}-\overbar{\bb} \vert^2  \lesssim \rho^{d-1}.
\end{equation}
Thus, by the triangle inequality, we have
\begin{equation}\label{eq:estimubtilde}
 \int_{\underline{Q_{\rho}}} \vert \nabla \tilde{u}\vert + \int_{Q_{\rho}}   \vert\tilde{\bb} - \shift \be_d + \nabla v_0  \vert^2  
 \lesssim\int_{Q_{\rho}}  \vert\overline{\bb} - \shift \be_d + \nabla v_0  \vert^2 + \rho^{d-1}.
\end{equation}
Furthermore, we note that $(u,\bb)$ is almost a minimizer in $\mathcal{A}^{\bb\cdot \nu}(Q_{\rho})$ (up to the interfacial energy concentrated in $\partial (\underline{Q_\rho})$), so that we have 
\[ 
     E(u,\bb,Q_\rho) \leq E(\tilde{u},\tilde{\bb},Q_\rho) + C\rho^{d-2}.
\]  
Applying Lemma \ref{orthogonality_lemma} to substract $-\nabla v_0 + \shift \be_d$ from the fields, this turns into
\[
     \int_{\underline{Q_{\rho}}} \vert \nabla u\vert + \int_{Q_{\rho}}  \Frac{1}{2} \vert \bb - \shift \be_d + \nabla v_0  \vert^2  \leq 
 \int_{\underline{Q_{\rho}}} \vert \nabla \tilde{u}\vert + \int_{Q_{\rho}}  \Frac{1}{2} \vert\tilde{\bb} - \shift \be_d + \nabla v_0  \vert^2  +C\rho^{d-2}.
\]
Combining the last estimate with \eqref{eq:estimubtilde} yields
\begin{equation}\label{eq:bvsbbar}
 \int_{\underline{Q_{\rho}}} \vert \nabla u\vert + \int_{Q_{\rho}}  \vert \bb - \shift \be_d + \nabla v_0  \vert^2  
\lesssim \int_{Q_{\rho}} \vert\overline{\bb} - \shift \be_d + \nabla v_0  \vert^2 + \rho^{d-1}.
\end{equation}

\medskip

We now crucially use Lemma \ref{shiftnonlin} to estimate the minimum energy of the relaxed problem in $Q_\rho$.
By the choice of $\rho$, we have
\begin{equation*}
\int_{\Gamma_\rho} \left ( ( \bb -\shift \be_d )\cdot \nu\right )^2  
\overset{\eqref{eq:rhofluxestim}}{\lesssim} l^{d-1} F(\shift,l)\leq \delta l^{d-1}.
\end{equation*}
Thus for sufficiently small $\delta$,  by Lemma \ref{shiftnonlin},
\begin{equation}\nonumber
\int_{Q_{\rho}} \vert \overbar{\bb}-\shift \be_d +\nabla v_0\vert^2
\lesssim \left ( \int_{\Gamma_{\rho}}\left ( (\overbar{\bb} -\shift \be_d) \cdot \nu\right ) ^2 \right )^{\frac{d}{d-1}}
\overset{\eqref{eq:rhofluxestim}}{\lesssim}  l^d F(\shift,l)^{\frac{d}{d-1}},
\end{equation}
 so that, restricting to the cube $Q_{\theta l}\subset Q_\rho$, we get
\begin{align}\label{estim2}
 \int_{\underline{Q_{\theta l}}} \vert \nabla u \vert + \int_{Q_{\theta l}} \vert \bb-\shift \be_d +\nabla v_0\vert^2 
&\leq  \int_{\underline{Q_{\rho}}} \vert \nabla u \vert +\int_{Q_{\rho}} \vert \bb-\shift \be_d +\nabla v_0\vert^2  \nonumber\\
&\hspace{-2cm}\overset{\eqref{eq:bvsbbar}}{\lesssim} \int_{Q_{\rho}} \vert \overbar{\bb}-\shift \be_d + \nabla v_0 \vert^2 +  l^{d-1} \lesssim  l^d F(\shift,l)^{\frac{d}{d-1}} + l^{d-1}.
\end{align}

\medskip

We may now conclude: inserting the estimates \eqref{estim1} and \eqref{estim2} into \eqref{decomp},  we obtain
\begin{equation*}
F(\tilde{\shift}, \theta l) \leq C\left ( \theta^2 F(\shift, l) + \dfrac{1}{\theta^{d}}F(\shift,l)^{\frac{d}{d-1}}+ \dfrac{1}{\theta ^{d}l}\right ).
\end{equation*}
We first choose $\theta \ll 1$ such that
\begin{equation*}
C \theta^2 F(\shift,l) \leq \frac{1}{2}\theta F(\shift,l)
\end{equation*}
and then $\delta \ll 1$ such that
\begin{equation*}
\dfrac{C}{\theta ^{d}} F(\shift,l)^{\frac{d}{d-1}} \leq  \frac{1}{2} \theta F(\shift,l),
\end{equation*} 
which, combined with \eqref{eq:shiftdiffestim} yields \eqref{eq:itstep}.
\end{proof}
\begin{proof}[Proof of Lemma \ref{iter}]
Let $\delta\in(0,+\infty)$ be arbitrary, to be chosen later. Consider a  minimizer $(u,\bb)$ in $\Aper$. Fix $\theta\in (0,1/2\,]$ according to Lemma \ref{itstep}. Without loss of generality, we ask that $l = \theta^N L$ for some positive integer $N$.
By induction over $n$, we will prove that there exists a sequence of shifts $\{\shift_n\}_{n=1}^N$ in $[\,-1/2,1/2\,]$, such that for $n=1,\dots, N$, there holds
\begin{equation}\label{F_ind_step}
F(\shift_n, \theta^{n}L) \leq  \theta^{n}  F(0,L)  + \Frac{C\theta^{n+1}}{L} \sum_{m=1}^{n} \theta^{-2m}
\ \  \text{and} \ \ 
 |\shift_n| \leq C\left ( F(0,L)^{1/2}  +\Frac{1}{\sqrt{\theta^{n} L}}\right),
\end{equation}
with  $C\in (0,+\infty)$ to be chosen later.
In particular we remark that for $n\leq N$, the first part of \eqref{F_ind_step} implies 
\begin{eqnarray}
F(\shift_n, \theta^{n}L) &\leq& \theta^{n}  F(0,L) + \Frac{C\theta^{n+1}}{L} \dfrac{\theta^{-2(n+1)}-1 }{\theta^{-2}-1}\nonumber \\
&=&  \theta^{n}  F(0,L) + \Frac{C\theta^{n+1}}{L} \Frac{\theta^{-2n}-\theta^2}{1-\theta^{2}}\nonumber\\
&\leq&  \theta^{n}  F(0,L) + \Frac{C}{1-\theta^{2}} \Frac{1}{\theta^{n-1}L}. \label{eq:Fcontrol}
\end{eqnarray}

\medskip

Letting $\shift_0 \defeq 0$, the inequalities in \eqref{F_ind_step} hold trivially for $n=0$. We now pass from $n$ to $n+1$.  Suppose that \eqref{F_ind_step}  and \eqref{eq:Fcontrol} hold at all steps from $0$ to $n$ with  $0\leq n \leq N-1$. We note that \eqref{F_ind_step} and \eqref{eq:Fcontrol} imply that the assumption \eqref{eq:iteration-condition} from Lemma~\ref{itstep} is satisfied provided $\delta$ in the present proof is chosen sufficiently small. Denoting by $C_0$ the implicit constant in \eqref{eq:itstep}, there exists $\shift_{n+1}$ such that
\begin{equation}\nonumber
F(\shift_{n+1}, \theta^{n+1}L) \leq \theta F(\shift_n,\theta^{n}L) + \Frac{C_0}{\theta^{n}L}
\!\overset{\eqref{F_ind_step}\text{ at step }n}{\leq} \theta^{n+1}  F(0,L) +\Frac{C\theta^{n+1} }{ L} \sum_{m=1}^{n} \theta^{-2m} + \Frac{C_0}{\theta^{n}L}.
\end{equation}
This is  consistent with \eqref{F_ind_step} provided we choose $C$ such that $C\geq C_0$. Furthermore, by \eqref{eq:itstep}, we have
\begin{equation}\label{eq:shiftdiff}
 |\shift_{n+1}-\shift_{n}| \leq C_0 F(\shift_n,\theta^{n}L)^{\frac{1}{2}}.
\end{equation}
We thus control $\shift_{n+1}$ by 
\begin{eqnarray*}
 |\shift_{n+1}| &\leq& \sum_{m=0}^n \vert \shift_{m+1} -\shift_m\vert \\
&\overset{\eqref{eq:shiftdiff}\text{ at every step}}{\leq}& C_0 \sum_{m=0}^n  F(\shift_m,\theta^{m}L)^{\frac{1}{2}}\\
&\overset{\eqref{eq:Fcontrol}\text{ at every step}}{\leq}& C_0  \left( \sum_{m=0}^n \left (\theta^{m}  F(0,L)\right )^{\frac{1}{2}} +  \sqrt{ \Frac{C}{1-\theta^{2}}}\sum_{m=1}^n  \Frac{1}{\sqrt{\theta^{m-1}L}}\right)\\
&\leq&  C_0  \dfrac{1}{1-\sqrt{\theta}} F(0,L)^{\frac{1}{2}} + C_0\sqrt{\dfrac{C}{1-\theta^{2}}}  \dfrac{1}{\sqrt{\theta}^{\ -1}-1}\  \dfrac{1}{ \sqrt{\theta^n L}}.
\end{eqnarray*}
This is consistent with \eqref{F_ind_step} at step $n+1$ and implies in particular $\shift_{n+1}\in [\,-1/2,1/2\,]$ provided we choose $C$ such that
\[ 
     C \geq \max\left \{C_0  \dfrac{1}{1-\sqrt{\theta}} , C_0\sqrt{\dfrac{C}{1-\theta^{2}}}  \dfrac{1}{\sqrt{\theta}^{\ -1}-1} \right \}.
\]

\medskip

Using \eqref{eq:Fcontrol} for $n=N$, we obtain
\begin{equation}\label{eq:Fcontrol2}
   F(\shift_N,l) \lesssim F(0,L) + \dfrac{1}{l}.
\end{equation}
By the triangle inequality, using the fact that $F$ is a volume average, we may remove the shift:
\begin{equation*}
 F(0,l) \lesssim F(\shift_N,l) + \shift_N^2 \overset{\eqref{eq:Fcontrol2},\eqref{F_ind_step}}{\lesssim} F(0,L) + \dfrac{1}{l}.
\end{equation*}
\end{proof}
\begin{proof}[Proof of Lemma \ref{shiftnonlin}]
We first note that $\bb$ being a minimizer of the relaxed problem \eqref{eq:relaxed} implies that for $\shift\in \R$,
 $\bb-\shift \be_d$ is a minimizer of the 
\textit{shifted relaxed problem}: 
\begin{align}\label{eq:shiftrelaxed}
 \min \left \{
\int_{\bulk} \Frac{1}{2} \vert \tilde\bb \vert^2 \ 
\middle \vert \begin{array}{cccc}
\Div \tilde\bb &=&0  &\text{ in  } Q_L,\\
\tilde\bb \cdot \nu &\in&[\,-1+\shift, 1+\shift\,] &\text{ on  } \underline{Q_L},\\
\tilde\bb\cdot \nu &=&(\bb -\shift\be_d)\cdot\nu  &\text{ on  } \Gamma_L.
\end{array}
\right \}
\end{align}
Indeed, if $\tilde{\bb}$ is a candidate for the problem \eqref{eq:shiftrelaxed}, writing $\be_d = \nabla x_d$, we obtain by integration by parts that $\int_{Q_L} (\bb -\shift \be_d -\tilde{\bb})\cdot \shift \be_d = 0$ so that
\[
      \int_{Q_L} \bb\cdot(\shift \be_d)-\dfrac{1}{2} \vert \shift \be_d\vert^2= \int_{Q_L} \tilde{\bb}\cdot(\shift \be_d)+\dfrac{1}{2} \vert \shift \be_d\vert^2
\]
to the effect of
\[
    \int_{Q_L} \dfrac{1}{2}\vert \tilde{\bb}\vert^2 - \int_{Q_L} \dfrac{1}{2} \vert \bb -\shift \be_d \vert^2  = \int_{Q_L} \dfrac{1}{2}\vert \tilde{\bb}+\shift \be_d\vert^2 - \int_{Q_L} \dfrac{1}{2} \vert \bb \vert^2.
\]
The last term is non-negative as $\tilde{\bb}+\shift\be_d$ is a candidate for the relaxed problem \eqref{eq:relaxed}, of which $\bb$ is a minimizer.
Thus, recalling that $\shift \in [\,-1/2,1/2\,]$ and considering a minimizer $\tilde{\bb}$ of the more constrained problem 
\begin{align*}
\min \left \{
\int_{\bulk} \Frac{1}{2} \vert \tilde\bb \vert^2 \ 
\middle \vert \begin{array}{cccc}
\Div \tilde\bb &=& 0  &\text{  in  } \bulk,\\
 \tilde\bb \cdot \nu &\in& [\,-1/2, -1/2\,]  &\text{ on  } \infbnd,\\
\tilde\bb\cdot \nu &=& (\bb -\shift\be_d)\cdot\nu &\text{  on  } \Gamma_L,
\end{array}
\right \}
\end{align*}
 there holds
\begin{equation}\label{eq:shiftvshalf}
     \int_{Q_L} \vert \bb - \shift \, \be_d \vert ^2 \leq \int_{Q_L} \vert \tilde{ \bb} \vert^2.
\end{equation}

\medskip

Consider also the solution $-\nabla v_0$ to the over-relaxed problem \eqref{eq:overrelaxed} in $Q_L$ with flux boundary data given by $\tilde{\bb}\cdot \nu$ on $\Gamma_L$. For $\tilde{\bb}$ and $-\nabla v_0$ it is clear that Lemma \ref{orthogonality_lemma}, but also version of Lemma \ref{nonlinestimate} hold (replacing the constraint $\bb\cdot \nu \in[\,-1,1\,]$ by $\bb\cdot \nu \in[\,-1/2,1/2\,]$  only affects the implicit constant), so that 
\[
     \int_{Q_L} \vert \tilde{\bb} \vert^2 -\int_{Q_L}  \vert \nabla v_0 \vert^2 = \int_{Q_L} \vert \tilde{\bb} + \nabla v_0 \vert^2 \lesssim \left(\int_{\Gamma_L} \big ((\bb - \shift \, \be_d )\cdot\nu \big) ^2\right )^{\frac{d}{d-1}}.
\]
Hence, using \eqref{eq:shiftvshalf} and once more Lemma \ref{orthogonality_lemma}, we get as desired
\begin{align*}
 \int_{Q_l} \vert \bb - \shift \, \be_d + \nabla v_0\vert ^2 
 \lesssim \left (\int_{\Gamma_l} \big(( \bb - \shift \, \be_d )\cdot\nu \big) ^2\right )^{\frac{d}{d-1}}.
\end{align*}
\end{proof}
\begin{proof}[Proof of Lemma \ref{nonconv}]
Let $\overline{\bb}$ be a minimizer for the relaxed problem \eqref{eq:relaxed} in $Q_L$ with flux boundary condition $g$.  Divide $Q_L$ into $\mathcal{O}(L^{d-1})$ small cubes $Q_i$ with side length $l\in [\,1,2\,]$ sitting on the bottom
and a large box above of height $L-l$ and sides of length $L$. Divide each $\underline{Q_i}$ into two boxes $\underline{Q_i^+}$ and $\underline{Q_i^-}$, such that $\text{area}(\underline{Q_i}^+ ) - \text{area}(\underline{Q_i}^- ) =\int_{\underline{Q_i}} \overline{\bb}\cdot \nu \in [\,-l^{d-1},l^{d-1}\,]$. Then define $u$ on $\underline{Q_i}$ so that
\begin{equation}\label{eq:chargecompat}
 u = \begin{cases} 
\ \, \,1 &\text{ on }  \underline{Q_i^+}, \\
\, -1 &\text{ on }  \underline{Q_i^-}\, ,
\end{cases}
\quad \text{and} \quad
\int_{\underline{Q_i}} u = \int_{\underline{Q_i}} \overline{\bb}\cdot\nu.
\end{equation}
 The interfacial energy in $Q_i$ is given by $\int_{\underline{Q_i}} |\nabla u| = 2 l^{d-2}$ (or zero if $\int_{\underline{Q_i}} \overline{\bb}\cdot \nu = \pm l^{d-1}$).  Taking into account the interface between the cubes $\underline{Q_i}$, the global interfacial energy in $\underline{Q_L}$ thus satisfies
\begin{align*}   
   \int_{\underline{Q_L}} \vert \nabla u \vert \leq  \Big (\dfrac{L}{l}\Big )^{d-1} 2l^{d-2} + 2d \dfrac{L}{l}  \ \ 2 L^{d-2} \lesssim L^{d-1}.
\end{align*}

It remains to modify the field $\overline{\bb}$ so that it becomes compatible with $u$. In each $Q_i$, we solve the following Neumann problem (which is solvable by \eqref{eq:chargecompat})
\begin{align*}
 \begin{cases}
  -\Delta v_i &=0  \qquad\qquad\ \text{ in } Q_i, \\
 -\nabla  v_i \cdot \nu &= 0 \qquad\qquad\ \text{ on } \partial Q_i\setminus \underline{Q_i}, \\
 -\nabla v_i \cdot \nu &=  u - \overline{\bb}\cdot \nu \quad\text{ on } \underline{Q_i}.
 \end{cases}
\end{align*}
As $u= \pm 1$ and $\vert \overline{\bb}\cdot\nu \vert \leq 1$ on $\underline{Q_i}$, we may apply Lemma \ref{harm_building_block} and there holds $ \int_{Q_i} \vert \nabla v_i\vert^2\lesssim 1$. Define a field $\bb$ by
\begin{equation*}
 \bb := \begin{cases}
  \overline{\bb} -\nabla v_i &\text{ on } Q_i,\\
   \overline{\bb}  &\text{ in } Q_L\backslash \bigcup_i Q_i.
   \end{cases} 
 \end{equation*}
By construction, $(u,\bb)\in \mathcal{A}^{g}(Q_L)$ and there holds
\begin{align*}
    \int_{Q_L}  \vert \bb -\overbar{\bb} \vert ^2=  \sum_i \int_{Q_i} \vert \nabla v_i\vert ^2 \lesssim L^{d-1}.
  \end{align*}
\end{proof}
\begin{proof}[Proof of Lemma \ref{nonlinestimate}]
Clearly, $\bb$ and $v_0$ satisfy the assumptions of Lemma \ref{orthogonality_lemma},
so we have 
\[
\int_{Q_L} \Frac{1}{2} \vert \bb + \nabla v_0 \vert^2 
= \int_{Q_L} \Frac{1}{2} \vert \bb \vert^2 -\int_{Q_L} \Frac{1}{2} \vert \nabla v_0 \vert^2 = E_{rel}^g(Q_L)-E_0^g(Q_L).
\]
Let $v$ be the minimizer of the dual to the relaxed problem on $Q_L$, see \eqref{eq:dualtorel}, and
let $w \defeq v-v_0$. There holds
\begin{align*}
E_{rel}^g(Q_L)-E^g_0(Q_L) &= -\int_{Q_L} \Frac{1}{2} \vert \nabla v \vert^2  - \int_{\underline{Q_L}} \vert v \vert + \int_{\Gamma_L} vg -\int_{Q_L} \Frac{1}{2} \vert \nabla v_0 \vert^2\\ 
&= -\int_{Q_L} \Frac{1}{2} \vert \nabla v - \nabla v_0 \vert^2 - \int_{Q_L} \nabla v \cdot \nabla v_0 - \int_{\underline{Q_L}} \vert v \vert +\int_{\Gamma_L}vg.
\end{align*}
We now integrate by parts the mixed term and appealing to \eqref{eq:overrelaxed}, we get
\begin{align*}
E_{rel}^g(Q_L)-E^g_0(Q_L) = -\int_{Q_L} \Frac{1}{2} \vert \nabla w \vert^2 + \int_{\underline{Q_L}} \left (-\vert w \vert + w\, \partial_d v_0\right ).
\end{align*}
Next, denoting $V \defeq (\int_{\underline{Q_L}} ( \partial_d v_0 )^2)^{1/2}$, by H\" older's inequality, Lemma \ref{interpolation} for $\epsilon\in(0,1\,]$, and Young's inequality, we obtain
\begin{align} \label{est_nonlin}
E^g_{rel}(Q_L)-E^g_0 (Q_L)
&\leq-\int_{Q_L} \Frac{1}{2} \vert \nabla w \vert^2 - \int_{\underline{Q_L}} \vert w \vert \,  +  \left ( \int_{\underline{Q_L}} w ^2\right)^{\frac{1}{2}} V 
\nonumber \\
&\hspace{-2cm}\leq -\int_{Q_L} \Frac{1}{2} \vert \nabla w \vert^2 +  C  (\epsilon L)^{\frac{1}{2}}  \left (\int_{Q_L} \vert \nabla w \vert^2 \right )^{\frac{1}{2}}V
+ \left(\dfrac{C V}{(\eps L)^{\frac{d-1}{2}}} -1 \right )\int_{\underline{Q_L}} \vert w \vert \nonumber\\
&\leq \dfrac{1}{2}C^2\eps LV^2  +\left(\dfrac{CV}{ (\eps L)^{\frac{d-1}{2}}} -1 \right )\int_{\underline{Q_L}} \vert w \vert.
\end{align}
From Lemma \ref{Neumann} applied to $v_0$ and the assumption $\int_{\Gamma_L} g^2 \ll L^{d-1}$, which we rewrite as $\int_{\Gamma_L} g^2 \leq cL^{d-1}$ for some $c\in(0,+\infty)$, we know that 
\begin{equation}\label{eq:defV}
 V = \left  (\int_{\underline{Q_L}} (\partial_d v_0 )^2\right )^{\frac{1}{2} } \lesssim  \left (\int_{\Gamma_L} g^2\right )^{\frac{1}{2}} \lesssim c^{\frac{1}{2}} L^{\frac{d-1}{2}}.
\end{equation}
Hence, if $c$ is small enough, depending only on $d$, we may choose $\epsilon\in(0,1\,]$ such that $ (\epsilon L)^{(d-1)/2}\geq CV$, so that the second term on the right hand side of \eqref{est_nonlin} is non-positive. Using \eqref{eq:defV} in \eqref{est_nonlin}, we thus obtain as desired
\begin{align*}
E^g_{rel}(Q_L)-E^g_0(Q_L) \lesssim V^{\frac{2}{d-1}+2}  \lesssim \left (\int_{\Gamma_l} g^2\right )^{\frac{d}{d-1}}.
\end{align*}
\end{proof}
\begin{proof}[Proof of Lemma \ref{orthogonality_lemma}]
  Using the condition $\nabla \cdot \bb=0$ to integrate by parts, we get
\begin{equation*}
 \int_{Q_L} \bb \cdot \nabla v_0 =- \int_{\Gamma_L} v_0  \bb\cdot \nu = -\int_{\Gamma_L} v_0 g =  \int_{\Gamma_L} v_0  \nabla  v_0\cdot \nu
 = - \int_{Q_L} \vert \nabla v_0 \vert^2.
\end{equation*}
Hence
\begin{equation*}
 \int_{Q_L} \Frac{1}{2} \vert \bb + \nabla v_0 \vert^2 = 
  \int_{Q_L} \Frac{1}{2} \vert \bb \vert^2 + \int_{Q_L} \bb \cdot \nabla v_0 + \int_{Q_L} \Frac{1}{2} \vert \nabla v_0 \vert^2
 = \int_{Q_L} \Frac{1}{2} \vert \bb \vert^2 - \int_{Q_L} \Frac{1}{2} \vert \nabla v_0 \vert^2.
\end{equation*}\qedhere
\end{proof}
\begin{proof}[Proof of Lemma \ref{dual}]
This proof is similar to that of Lemma 3.3 in \cite{ACO09}. The idea is to replace the constraints by a linear (thus concave) maximization problem. When the constraints for $\bb$ are not met, the supremum is infinite and such candidates cannot be minimizers. The condition $\vert\int_{\Gamma_L} g\vert\leq L^{d-1}$ ensures that the class on which we minimize for the relaxed problem is not empty. This minimization problem can then be stated as
\begin{align*}
E^{g}_{rel}(Q_L) &= \inf_{\bb} \left \{ \int_{Q_L} \Frac{1}{2} \vert \bb \vert ^2 \,\middle \vert\,
\begin{array}{cccc}
\nabla \cdot \bb &=& 0 &\text{in } Q_L,\\
 \bb \cdot \nu &\in& [\,-1,1\,] &\text{on } \underline{Q_L},\\
\bb \cdot \nu &=& g &\text{on } \Gamma_L 
\end{array}
\right \}
\\ &= \inf_{\bb,u}\left \{ \int_{Q_L} \Frac{1}{2} \vert \bb \vert ^2 \,\middle \vert\,
\begin{array}{ccc}
\bb &\in& \mathrm{L}^2(Q_L,\R^d),\\
u &\in& \mathrm{L}^{2}(\underline{Q_L}, [-1,1]),
\end{array}
\begin{array}{ccc}
\nabla \cdot \bb &= 0 &\text{ in } Q_L,\\
\bb \cdot \nu &= u &\text{ on } \underline{Q_L},\\
\bb \cdot \nu &= g &\text{ on } \Gamma_L 
\end{array}
\right \}
\\ &\hspace{-1cm}= \inf_{\bb,u}\sup_{v}\left \{ \int_{Q_L} \left (\Frac{1}{2} \vert \bb \vert ^2 +\bb \cdot \nabla v \right) + \int_{\underline{Q_L}}  v u     + \int_{\Gamma_L} v g \,\middle \vert\,
\begin{array}{ccc}
\bb &\in &\mathrm{L}^2(Q_L,\R^d),\\
u &\in &\mathrm{L}^{2}(\underline{Q_L}, [-1,1]),\\
v &\in &\mathrm{H}^1(Q_L).
\end{array}
\right \}
\end{align*}
The infimum being taken for a convex functional, and the supremum for an affine one, we use a classical $min-max$ theorem (see for instance \cite{BrezisOpMaxMon}, Chapter I, Proposition 1.1) to change the order of the operations. To see that we may apply this proposition, let us denote the functional in the last line by $K((u,\bb),v)$; it is defined on the product of the spaces $E\defeq \mathrm{L}^{2}(\underline{Q_L})\times \mathrm{L}^2(Q_L,\R^d)$ and $\mathrm{H}^1(Q_L)$, equiped with the weak topology. Consider the convex subset $A\defeq E\cap\{(u,\bb), u \in[\,-1,1\,]\}$. The functional $K$ is linear and bounded in the variable $v$, thus concave and weakly continuous. It is also convex and weakly lower semi continuous in the variable $(u,\bb)$ on $A$. If $\widehat{\bb}=-\nabla \widehat{v}$ is the minimizer for the relaxed problem (the fact that the minimizers are gradients is standard, as for the over-relaxed problem \eqref{eq:overrelaxed}, see Section IX.3 of \cite{DautLionBook3}) then the set
\[
    \left \{ (u,\bb)\in A\ \middle \vert \  K((u,\bb),\widehat{v}) \leq  E^{g}_{rel} (Q_L)\right \}
\]
is not empty (take the pair $(\widehat{\bb}\cdot \be_d, \widehat{\bb})$). It is also weakly compact, by the weak lower semi continuity of $K$ in the variable in the variable $(u,\bb)$ and the condition $u\in [\,-1,1\,]$. Thus, by the aforementioned proposition, we have 
\[
     \inf_{(u,\bb)\in A} \sup_{v\in \mathrm{H}^1(Q_L)} K((u,\bb),v) = \sup_{v\in \mathrm{H}^1(Q_L)} \inf_{(u,\bb)\in A}  K((u,\bb),v).
\]

\medskip
 
Furthermore, as $\vert \bb \vert ^2 +2 \bb \cdot \nabla v= \vert\bb + \nabla v\vert^2 - \vert \nabla v \vert^2$, the infimum is reached for $\bb = -\nabla v$ in $Q_L$ and $u = -\mathrm{sgn} \, v$ on $\underline{Q_L}$, thus
\begin{align*}
E^{g}_{rel}(Q_L)&=\\
&\hspace{-1cm} \sup_{v}\inf_{\bb,u} \left \{ \int_{Q_L} \left (\Frac{1}{2} \vert \bb \vert ^2 +\bb \cdot \nabla v \right) +\int_{\underline{Q_L}}  v u + \int_{\Gamma_L} v g \,\middle \vert\,
\begin{array}{ccc}
\bb &\in& \mathrm{L}^2(Q_L,\R^d),\\
u &\in& \mathrm{L}^{2}(\underline{Q_L}, [-1,1]),\\
v &\in& \mathrm{H}^1(Q_L)
\end{array}
\right \}
\\
&= \sup_{v}\left \{ -\int_{Q_L}  \Frac{1}{2} \vert \nabla v \vert^2 - \int_{\underline{Q_L}}  \vert v\vert  + \int_{\Gamma_L} v g \,\middle\vert\, v \in \mathrm{H}^1(Q_L) 
\right \}
\\
&= - \inf_v\left\{ \int_{Q_L} \Frac{1}{2} \vert \nabla v \vert ^2 + \int_{\infbnd} \vert v \vert - \int_{\Gamma_L} vg \,\middle \vert\, v \in \mathrm{H}^1(Q_L) \right \}.
\end{align*}
\end{proof}
\begin{proof}[Proof of Lemma \ref{interpolation}]
By scaling and translation invariance, it is sufficient to consider the cube $(0,1)^d$. Let $w$ be a smooth function on this cube $(0,1)^d$, clearly
\begin{align*}
\vert w (x',x_d) \vert \leq \int_0^{x_d} \vert \nabla w (x',t)\vert \dd t + \vert w(x',0) \vert.
\end{align*}
Integrating over $x_d$ between $0$ and $1$, we get
\begin{align*}
\int_0^1 \vert w (x',x_d) \vert \dd x_d \leq
\int_0^1 \vert \nabla w (x',t)\vert \dd t + \vert w(x',0) \vert. 
\end{align*}
Integrating over $x'\in(0,1)^{d-1}$ and using 
H\" older's inequality on the first term on the right hand side now yields
\begin{align*} 
\int_{(0,1)^d} \vert w \vert \leq \left ( \int_{(0,1)^d} \vert \nabla w\vert ^2 \right )^{\frac{1}{2}} + \int_{(0,1)^{d-1}\times \{0\}} \vert w \vert.
\end{align*}
Next, from Lemma 3.2 in \cite{ACO09}, with $\eps \defeq 1$ we know that
\begin{align*}
\int_{(0,1)^{d-1}\times \{0\}} w^2 \lesssim \int_{(0,1)^d} \vert \nabla w \vert^2  + \left( \int_{(0,1)^d} \vert w \vert\right)^2,
\end{align*}
which, combined with the square of the previous estimate yields
\begin{equation}\label{eq:scale1}
\int_{(0,1)^{d-1}\times \{0\}}  w ^2  \lesssim \int_{(0,1)^d} \vert \nabla w \vert^2 + \left(\int_{(0,1)^{d-1}\times \{0\}} \vert w \vert\right)^2 .
\end{equation}
By approximation and trace estimate, \eqref{eq:scale1} remains true for any $w\in \mathrm{H}^1(Q_L)$.

\medskip

Now given a function $w$ on $(0,1)^d$ and $\eps \in (0,1\,]$, with $\epsilon =1/N$ for some positive integer $N$, we divide the lower side of the cube into $N^{d-1}$ boxes of side length $\eps$. 
In each of these boxes (letting $x^i$ correspond to the centers of their bases) inequality \eqref{eq:scale1} is applied to the rescaled potential $w(x^i+\eps x)$. Using the fact that the sum of squares of non-negative numbers is smaller than or equal to the square of their sum, we obtain
\begin{align*}
\int_{(0,1)^{d-1}\times \{0\}}  w ^2 \lesssim \eps \int_{(0,1)^{d-1}\times (0,\epsilon)} \vert \nabla w \vert^2 + \Frac{1}{\eps^{d-1}}\left (\int_{(0,1)^{d-1}\times \{0\}} \vert w \vert \right )^2.
\end{align*}
Extending the first integral on the right hand side to the whole of $(0,1)^d$ and taking the square root, we get the desired estimate. It is easily seen that up to a change of constant, this estimate holds for any $\epsilon \in (0,1\,]$ and not just for inverses of positive integers.
\end{proof}
\begin{proof}[Proof of Lemma \ref{Neumann}]
A similar result is established by a somewhat different
argument in \cite[Remark 5.5]{MiuraOttoarXiv}. 
We start with a couple of reductions. By scaling invariance, it is enough to prove the estimate with $(0,L)^{d}$ replaced by $(0,\pi)^d$. Decomposing the harmonic function $v$ into $2^d-1$ parts, we may restrict ourselves to the situation where $v$ has zero boundary flux on all but two of the $2^d$ faces of the cube $(0,\pi)^d$, which we call {\it input face} and {\it output face}, and zero Dirichlet boundary condition on the output face (the one on which we want to estimate the $\mathrm{L}^2$-norm of the normal derivative).
We suppose $\mathrm{L}^2$-control of the normal derivative on the input face. We distinguish two cases: The easier case in which the input face is {\it opposite} to the output face
and the harder case in which they are
{\it adjacent}.
By cubic symmetry, we may in both cases take the input face
to be the top face $\{x_d=\pi\}$.
In the easy case, we have
\begin{equation} \label{o1}
v(x',0)=0 \text{ for } x'\!\in (0,\pi)^{d-1}\   \text{and}\ \ 
\partial_i v(x)=0  \text{ if }  x_i\!\in\!\{0,\pi\}
\text{ for some } i= 1,\dots,d-\!1. 
\end{equation}
We then seek the following estimate on the bottom face $\{x_d=0\}$:
\begin{equation}\label{o2}
\int_{(0,\pi)^{d-1}}(\partial_dv)^2(x_d=0) dx_1\cdots dx_{d-1}\lesssim
\int_{(0,\pi)^{d-1}}(\partial_dv)^2(x_d=\pi)dx_1\cdots dx_{d-1}.
\end{equation}
In the hard case, we can suppose that $v$ vanishes on the face $\{x_1=0\}$, we thus have
\begin{equation}\label{o1bis}
v(0,x_2,\dots,x_d) = 0\ \ \text{and}\ \ 
\partial_i v(x)=0  \text{ if }\begin{cases}
x_i\in\{0,\pi\} \text{ for some }i=2,\dots,d-1,\\ 
\text{or } x_1= \pi ,\\
\text{or } x_d= 0. \end{cases}
\end{equation}
We then seek the estimate
\begin{equation}\label{o3}
\int_{(0,\pi)^{d-1}}(\partial_1v)^2(x_1=0) dx_2\cdots dx_{d}\lesssim
\int_{(0,\pi)^{d-1}}(\partial_dv)^2(x_d=\pi)dx_1\cdots dx_{d-1}.
\end{equation}

\medskip

We will show both with help of Fourier series. In the easy case, in view of (\ref{o1}), we may develop $v$ in Fourier 
series in the horizontal variables
$x'=(x_1,\dots,x_{d-1})$;
because of the harmonicity of $v$ and of the fact that $v(\{x_d=0\}) = 0$, these
take on the form
\begin{equation}\nonumber
v=\sum_{n'}a_{n'}\cos(n_1x_1)\cdots\cos(n_{d-1}x_{d-1})\sinh(|n'|x_d),
\end{equation}
where the sum runs over all $n'\in\mathbb{N}^{d-1}$ and
$|n'|^2=n_1^2+\cdots+n_{d-1}^2$.
Because of
\begin{equation}\nonumber
\partial_dv=\sum_{n'}|n'|a_{n'}\cos(n_1x_1)\cdots\cos(n_{d-1}x_{d-1})\cosh(|n'|x_d)
\end{equation}
and Parseval's identity we may re-express (\ref{o2}) as
\begin{equation}\nonumber
\sum_{n'}|n'|^2a_{n'}^2 \lesssim \sum_{n'}|n'|^2\cosh^2(\pi|n'|)a_{n'}^2.
\end{equation}
As $\cosh \geq 1$, this holds and the easy case (\ref{o2}) follows.

\medskip

In the hard case, by \eqref{o1bis}, reflecting $v$ evenly across the plane $\{x_1=\pi\}$, we can write
\begin{equation}\nonumber
v=\sum_{n'}a_{n'}\sin((n_1+\textstyle{\frac{1}{2}})x_1)\cos(n_2x_2)\cdots\cos(n_{d-1}x_{d-1})\cosh(\alpha(n') x_d),
\end{equation}
where $\alpha(n')>0$ satisfies $ \alpha(n')^2= (n_1+\frac{1}{2})^2+n_2^2+\cdots+n_{d-1}^2$, and the sum runs over all $n'=(n_1,\dots,n_{d-1})\in\mathbb{N}^{d-1}$.
By Parseval's identity applied to the variables $x_2,\dots, x_{d-1}$ over $(0,\pi)$ and to the variable $x_1$ over $(0,2\pi)$ we may re-express the r.\ h.\ s. of (\ref{o3}) as
\begin{align}
\int_{(0,\pi)^{d-1}}(\partial_dv)^2(x_d=\pi)dx_1\cdots dx_{d-1}
&\sim \sum_{n'}\alpha(n')^2a_{n'}^2 \sinh^2(\pi \alpha(n'))\nonumber\\
&\sim \sum_{n'}\alpha(n')^2a_{n'}^2 \exp(2\pi \alpha(n')),\label{o5}
\end{align}
where the last comparison follows from the fact that $\alpha(n')\geq \frac{1}{2}$.
For the l.\ h.\ s. of \eqref{o3}, Parseval's identity applied to the variables $x_2,\dots, x_{d-1}$ yields
\begin{equation}\nonumber
\int_{(0,\pi)^{d-1}}\hspace{-.2cm}(\partial_1v)^2(x_1=0)dx_2\cdots dx_{d}\sim
\hspace{-.3cm}\sum_{n_2,\dots,n_{d-1}}
\int_0^\pi\hspace{-.1cm}\Big (\sum_{n_1}(n_1+\textstyle{\frac{1}{2}})a_{n'}\cosh(\alpha(n')x_d)\Big)^2 dx_d.
\end{equation}
We consider the individual terms on the r.\ h.\ s. and start by expanding the square
\begin{eqnarray*}
\int_0^\pi\Big(\sum_{n_1}(n_1+{\textstyle \frac{1}{2}}) a_{n'}\cosh(\alpha(n') x_d)\Big)^2dx_d&&\\
&&\hspace{-5cm}=\sum_{n_1}\sum_{m_1}(n_1+{\textstyle\frac{1}{2}})a_{n'}(m_1+{\textstyle \frac{1}{2}})a_{m'}\int_0^\pi\cosh(\alpha(n')x_d)\cosh(\alpha(m')x_d)dx_d,
\end{eqnarray*}
where, with a slight abuse of notation, $m':=(m_1,n_2,\dots,n_{d-1})$. 
Because of the elementary inequality
\begin{align*}
\int_0^\pi\cosh(\alpha(n')x_d)\cosh(\alpha(m')x_d)dx_d&\leq\int_0^\pi\exp(\alpha(n') x_d+\alpha(m') x_d)dx_d\\
&\leq\frac{1}{\alpha(n')+\alpha(m')}\exp(\pi(\alpha(n')+\alpha(m')))
\end{align*}
 and thus
\begin{eqnarray}
\int_{(0,\pi)^{d-1}}(\partial_1v)^2(x_1=0)dx_2\cdots dx_d \hspace{-5cm}&\nonumber\\
&\lesssim&\sum_{n_2,\dots,n_{d-1}}
\sum_{n_1}\sum_{m_1}(n_1+\textstyle{\frac{1}{2}})|a_{n'}|(m_1+\textstyle{\frac{1}{2}})|a_{m'}|\frac{1}{\alpha(n')+\alpha(m')}\exp(\pi(\alpha(n')+\alpha(m')))\nonumber\\
&\le&\sum_{n_2,\dots,n_{d-1}}
\sum_{n_1}\sum_{m_1}\frac{1}{n_1+m_1+1} \ \alpha(n')|a_{n'}|\exp(\pi \alpha(n'))\ \alpha(m')|a_{m'}|\exp(\pi\alpha(m')),\nonumber
\end{eqnarray}
as $\alpha(n') \geq n_1+ \frac{1}{2}$. 
A glance at (\ref{o5}) now shows that (\ref{o3}) reduces to the following statement
on non-negative
sequences $\{b_{n'}:=\alpha(n')\exp(\pi \alpha(n'))|a_{n'}|\}_{n'}$
\begin{equation}\nonumber
\sum_{n_2,\dots,n_{d-1}}
\sum_{n_1}\sum_{m_1}\frac{1}{n_1+m_1+1}b_{n'}b_{m'}\lesssim
\sum_{n_2,\dots,n_{d-1}}\sum_{n_1}b_{n'}^2,
\end{equation}
which clearly can be disintegrated into
\begin{equation}\label{o7}
\sum_{n=0}^\infty\sum_{m=0}^\infty\frac{1}{n+m+1}b_{n}b_{m}
\lesssim 
\sum_{n=0}^\infty b_{n}^2.
\end{equation}

\medskip

We conclude the proof with an argument for (\ref{o7}). By symmetry in $n$ and $m$,
the statement follows from
\begin{equation}\nonumber
\sum_{n=0}^\infty\sum_{m=0}^n\frac{1}{n+m+1}b_{n}b_{m}
\lesssim 
\sum_{n=0}^\infty b_{n}^2,
\end{equation}
and to reduces to
\begin{equation}\nonumber
\sum_{n=0}^\infty b_{n}\frac{1}{n+1}\sum_{m=0}^nb_{m}\lesssim \sum_{n=0}^\infty b_{n}^2.
\end{equation}
Applying the Cauchy-Schwarz inequality to the left hand side we see that \eqref{o5} follows directly from
\begin{equation}\label{o8}
\sum_{n=0}^\infty\frac{1}{(n+1)^2}\left (\sum_{m=0}^nb_{m}\right )^2\lesssim
\sum_{n=0}^\infty b_{n}^2.
\end{equation}
Let us prove this last estimate. In terms of the discrete anti-derivative $B_n:=\sum_{m=0}^nb_{m}$, (\ref{o8})
amounts to a discrete version of Hardy's inequality and can be established in a similar way: We obtain by a discrete integration by parts
\begin{eqnarray*}
\sum_{n=0}^N\frac{1}{(n+1)^2}B_n^2 &\sim& \sum_{n=0}^N\frac{1}{(n+1)(n+2)}B_n^2 = \sum_{n=0}^N\left (\frac{1}{n+1}-\frac{1}{n+2}\right )B_n^2\\
&\le& B_0^2+\sum_{n=1}^{N}\frac{1}{n+1}(B_{n}^2-B_{n-1}^2)= b_0^2+\sum_{n=1}^{N}\frac{1}{n+1}b_n(B_n+B_{n-1})\\
&\le& b_0^2+\left(\sum_{n=1}^{N}b_n^2\sum_{n=1}^N\frac{1}{(n+1)^2}(B_n+B_{n-1})^2\right)^\frac{1}{2}\\
&\le& b_0^2+2\left(\sum_{n=1}^{N}b_n^2\sum_{n=1}^N\frac{1}{(n+1)^2}B_n^2\right)^\frac{1}{2}
\end{eqnarray*}
and thus by Young's inequality $\sum_{n=0}^N\frac{1}{(n+1)^2}B_n^2\lesssim
\sum_{n=0}^{N}b_n^2$, which yields
(\ref{o8}) in the limit $N\uparrow\infty$.
\end{proof}
\begin{proof}[Proof of Lemma \ref{harm_building_block}]
Without loss of generality, we may assume that $L=1$ by scaling, and work in $(0,1)^d$. It also suffices to prove the statement for $p = 2(d-1)/d$ as the other cases follow from Jensen's inequality. By the Sobolev trace theorem \cite[Theorem~5.36]{AdamsSobolev2}, (which we may use because the cube is bi-Lipschitz equivalent to a ball,) and the Poincar\'e inequality (we may assume that $v$ has zero average,) there holds:
\begin{equation}\label{eq:traceestimate}
    \left (\int_{\partial(0,1)^d} \vert v\vert^{\frac{2(d-1)}{d-2}}\right )^{\frac{d-2}{2(d-1)}} \lesssim \left(\int_{(0,1)^d} \vert \nabla v\vert^2\right)^{\frac{1}{2}}.
\end{equation}
Noting that $q \defeq 2(d-1)/(d-2)$ satisfies $1/p + 1/q= 1$, by integration by parts and H\"older's inequality, we may write
\begin{equation*}
 \int_{(0,1)^d} |\nabla v|^2 = \int_{\partial (0,1)^d} v g \leq  \left(\int_{\partial (0,1)^d} \vert v\vert^q \right)^{\frac{1}{q}}  \left( \int_{\partial (0,1)^d} \vert g\vert ^p \right)^{\frac{1}{p}}. 
\end{equation*}
Using \eqref{eq:traceestimate} and regrouping the terms in $\nabla v$ yields the desired estimate.
\end{proof}

\medskip

 \begin{proof}[Proof of Lemma \ref{v_periodic}]
Considering $\bb$ on the strip $\R^{d-1}\times (0,L)$, the fact that it derives from a potential ($\bb = -\nabla \tilde{v}$) is standard (as for the over-relaxed problem \eqref{eq:overrelaxed}, see Section IX.3 of \cite{DautLionBook3}).
Moreover, we have $\tilde v(x)= v(x)+\xi'\cdot x'$, where $v$ is (horizontally) periodic (with period $L$) on $\mathbb{R}^d\times(0,L)$ and $\xi'\in\mathbb{R}^{d-1}$. The latter can be seen by noting that for any $i=1,\dots,d-1$, the function $\tilde v(\cdot+Le_i)-\tilde v$ has gradient $-\bb(\cdot+L\be_i)+\bb$, which vanishes by periodicity of $\bb$, and thus must agree with some constant $\xi_i$, which implies
that $v(x):=\tilde v(x)-\xi'\cdot x'$, where $\xi':=(\xi_1,\dots,\xi_{d-1})$, is periodic.
It remains to prove that $\xi'=0$. Since replacing $\bb$ by $\bb-\dashint_{[0,L)^{d-1}\times(0,L)}\bb'$, where $\bb'$ denotes the horizontal component of $\bb$, does not affect the constraint and strictly reduces the energy unless $\dashint_{[0,L)^{d-1}\times(0,L)}\bb'=0$, we learn from minimality that this average must indeed vanish. Taking the average of the identity  $-\bb'=\nabla' v+\xi'$ and noting that $\dashint_{[0,L)^{d-1}\times(0,L)}\nabla'v=0$, because the horizontal variables run over the torus, we obtain the desired $\xi'=0$.

\medskip

The fact that $v$ can be chosen to vanish on $\{x_d= L\}$ follows from the free boundary condition on that face, in the same way that the over-relaxed potential $v_0$ of \eqref{eq:overrelaxed} vanishes on $\underline{Q_L}$.
\end{proof}
\begin{proof}[Proof of Lemma \ref{hoelder_v}]
We first prove that $v$ is H\"older continuous on $\R^{d-1}$. By Corollary \ref{firstdecay} and Lemma \ref{small_scales}, it is clear that for $(x',x_d)\in \R^{d-1}\times (0,L)$, there holds
\begin{equation}\nonumber
  \vert \bb (x',x_d)\vert \lesssim \dfrac{1}{x_d\,^{1/2}}.
\end{equation}
Thus, considering two points $x',y' \in \R^{d-1}$, which by horizontal periodicity we may suppose to be at distance less than $L$, we join $(x',0)$ to $(y',0)$ by a vertical half circle contained in $\R^{d-1}\times (0,L)$ of diameter $2r = \vert x'-y'\vert$ and integrate $\nabla v= -\bb$ along this curve to get
\begin{equation}\nonumber
    \vert v(x',0)-v(y',0)\vert \leq \int_0^{\pi} \dfrac{1}{(r\sin s)^{1/2}} r\dd s \lesssim \sqrt{r}.
\end{equation}
Furthermore by Lemma \ref{v_periodic}, $v$ vanishes on $\R^{d-1}\times \{L\}$. Thus the potential $v$ solves a Dirichlet problem with H\"older-$1/2 $ boundary conditions on $\R^{d-1}\times \{0,L\}$. As can be seen via the representation through the Poisson kernel (for the slab), the modulus of H\"older continuity transmits (up to a constant) from the boundary data to its harmonic extension.
\begin{equation}\label{eq:vholder}
[\, v\,]_{\text{C}^{1/2}(Q_L)} \lesssim 1.
\end{equation}

\medskip

Let us turn to the over-relaxed potential $v_0$ defined on $Q_l$;  let $w\defeq v-v_0$. As $v_0=0$ on $\underline{Q_l}$, using the first part of the lemma, there holds $w=v$ on $\underline{Q_l}$ hence $[\, w\,]_{\text{C}^{1/2}(\underline{Q_l})} \lesssim 1$. Furthermore, since on $\Gamma_l$, $\nabla v_0 \cdot \nu = -\bb \cdot \nu=  \nabla v \cdot \nu$, there holds $\nabla w\cdot \nu =0$ on $\Gamma_l$. We thus reflect $w$ evenly 
 across $\Gamma_l \cap \{x_d=l\}$ to extend it harmonically onto a box of double the height. Subsequently by horizontal even reflections, we extend $w$ harmonically to the whole strip $\R^{d-1}\times [\,0,2l\,]$. We control the H\"older-$1/2$ semi-norm of $w$ on the whole boundary. We then conclude as in the case of $v$ that $[\, w\,]_{C^{1/2}(Q_l)}\lesssim 1$, which yields $[\, v_0\,]_{C^{1/2}(Q_l)}\lesssim 1$ by \eqref{eq:vholder} and the triangle inequality.
\end{proof}
\begin{proof}[Proof of Lemma \ref{pointwise_estimate}]
 As the potential $v_0$ vanishes on $\underline{Q}_l$,
its odd extension across the plane $\{x_d=0\}$, which we still denote by $v_0$, is harmonic.
Setting $r:={\rm dist}((x',0),\Gamma_l)$, we thus obtain by inner regularity theory (Theorem~2.10 in \cite{GilbargTrudinger})
\begin{align}\label{lo01}
|\partial_d v(x',0)|\lesssim\frac{1}{r}\sup_{B((x',0),r)}|v|.
\end{align}
By the H\"older continuity of Lemma \ref{hoelder_v} we obtain in particular
$|v_0(x)|\lesssim x_d\,^{1/2}$, so that (\ref{lo01}) yields the desired
estimate in the form of $|\partial_d v(x',0)|\lesssim r^{-1/2}$.
 \end{proof}
\begin{proof}[Proof of Corollary \ref{firstdecay}]
By periodicity, we may assume $x'=0$. For $L/2\geq x_d \gg 1$, by Theorem \ref{uniform_bound} there holds
\[
    \int_{Q_{2x_d}} \vert \bb \vert^2 \lesssim x_d\,^{d-1}.
\]
This extends to $L\geq x_d\gg 1$ by reflection across $\{x_d= L\}$, as $v=0$ on that plane.
Since $\vert \bb\vert^2$ is sub-harmonic in $Q_{2x_d}$, the mean value property yields
\[
    \vert \bb(0,x_d) \vert^2 \lesssim \dfrac{1}{x_d\,^d}\int_{Q_{2x_d}}\vert \bb \vert^2 \lesssim \dfrac{1}{x_d}.
\]
\end{proof}
\begin{proof}[Proof of Lemma \ref{small_scales}]
Consider a minimizer $(u,\bb)\in \Aper$. We want to control the field $\bb(x',x_d)=-\nabla v(x',x_d)$. Without loss of generality, we assume $x'=0$. 
Given $R>0$, let $B^+_R \defeq B(0,R) \cap \{x_d>0\}$ and $B_R^{d-1}\defeq B(0,R) \cap \{x_d=0\}$. We compare $v$ with the potential $\tilde{v}$ generated by the charges $u$ in $B_2^{d-1}$, which can be written as a single layer potential. We only use the representation of the gradient of $\tilde{v}$:
\begin{equation}\label{eq:nablav}
\nabla \tilde{v}(x',x_d) = c_d \int_{B_{2}^{d-1}} \Frac{(x'-y',x_d)}{\vert (x',x_d)-(y',0)\vert^d} u(y') \dd y'.
\end{equation}
The function $w\defeq v-\tilde{v}$ is harmonic in $B_2^+$ with zero boundary flux on $B_2^{d-1}$; it can thus be reflected across $B_2^{d-1}$ to obtain a harmonic function on $B_2$, which we still denote by $w$. 
Hence $\vert \nabla w\vert^2$ is sub-harmonic and there holds
\begin{equation}\label{eq:nablaw}
\sup_{x \in B_1^+}\vert \nabla w(x) \vert \lesssim \left (\int_{B_{2}} \vert \nabla w \vert^2 \right )^{\frac{1}{2}} \lesssim \left (\int_{B_{2}^+} \vert \nabla \tilde{v} \vert^2 + \vert \nabla v \vert^2 \right )^{\frac{1}{2}}.
\end{equation}
Applying Theorem \ref{uniform_bound} to $Q_{4}$, we get $\int_{B_2^+}\vert \nabla v\vert^2 \lesssim 1$, so that it remains to control $\vert \nabla \tilde{v}\vert^2$. We claim that for $(x',x_d)\in B_2^+$,
\begin{equation}\label{eq:tildev}
\vert \nabla \tilde{v}(x',x_d) \vert \lesssim \ln\left (\dfrac{1}{x_d}\right),
\end{equation}
which  first gives $\sup_{x\in B_1^+} \vert \nabla w(x)\vert \lesssim 1$ by \eqref{eq:nablaw} and then the statement of the lemma by the triangle inequality (in the sup-norm).

\medskip

The proof of estimate \eqref{eq:tildev} is elementary by the representation \eqref{eq:nablav}, which yields
\begin{eqnarray*}
  |\nabla \tilde{v}(x',x_d)| &\lesssim& \int_{B_{2}^{d-1}} \Frac{1}{(\vert x'-y'\vert +  x_d)^{d-1}} \dd y'
  \lesssim \int_{B_{4}^{d-1}} \Frac{1}{(\vert z'\vert +  x_d)^{d-1}} \dd z'\\
  &\lesssim& \int_0^4 \Frac{r^{d-2}}{(r +  x_d)^{d-1}} \dd r
   \lesssim \int_0^4 \Frac{1}{r +  x_d} \dd r \lesssim \ln \left(\dfrac{1}{x_d} \right ).
\end{eqnarray*}
\end{proof} 

\medskip

\begin{proof}[Proof of Lemma \ref{sigma_null_sigma}]
Suppose, as in the statement of the lemma, that $L\geq l\gg 1$, and that $(u,\bb)$ is a minimizer in $\Aper$. Let $v_0$ be the solution to the over-relaxed problem \eqref{eq:overrelaxed} with flux boundary data $\bb\cdot \nu$ on $\Gamma_l$. It suffices to construct a candidate  $(u^*, \bb^*)$ in $\mathcal{A}^0(Q_l)$ such that 
\begin{equation}\label{eq:ubstar}
   \dfrac{E(u^*,\bb^*,Q_l)}{l^{d-1}}\leq  \dfrac{E(u,\bb,Q_l)}{l^{d-1}} -\dfrac{1}{l^{d-1}} \int_{Q_l} \dfrac{1}{2}\vert  \nabla v_0\vert^2 +\dfrac{C}{l^{1/2}}.
\end{equation}
 The candidate field $\bb^*$ will be a controlled modification of $\bb+ \nabla v_0$, which has vanishing flux boundary condition on $\Gamma_l$. We need to modify it near $\underline{Q_l}$, jointly with $u$ in order to obtain the right flux boundary condition on $\underline{Q_l}$. We decompose the bottom of $Q_l$ into cubes $Q_i$ of side length $\lambda\sim 1$, such that $\lambda$ divides $l$ and is large enough (depending only on $d$, in view of an application of Lemma \ref{lemma_flow}) so that on each cube $Q_i$ there holds
\begin{equation}\label{estim_energy}
   E(u,\bb,Q_i) \lesssim \lambda^{d-1} \text{ (by Theorem \ref{uniform_bound}), or }E(u,\bb,Q_i) \leq\lambda^{d-2} \Lambda \text{ for some }\Lambda \lesssim 1,
  \end{equation}
  as well as  $\vert \dashint_{\underline{Q_i}} u \vert \leq \frac{1}{2}$ (by Lemma \ref{control_u_avg}), and  $ \dashint_{\underline{Q_i}} \vert \partial_d v_0\vert \leq \frac{1}{2}$ (by Lemma \ref{pointwise_estimate}).
Indeed, the last condition is satisfied for large enough $\lambda$, as, by Lemma~\ref{pointwise_estimate}:
\begin{equation}\nonumber
    \dashint_{\underline{Q_i}}\vert \partial_d v_0\vert \lesssim \dashint_{\underline{Q_i}} \dist((x',0), \Gamma_l)^{-\frac{1}{2}}\dd x' \lesssim \dfrac{1}{\lambda^{1/2}}.
\end{equation}

\medskip

For each cube $Q_i$ not adjacent to $\Gamma_l$, again by Lemma~\ref{pointwise_estimate}, we have:
\begin{equation}\label{eq:v0flux}
    \dashint_{\underline{Q_i}}\vert \partial_d v_0\vert \lesssim \dist(Q_i,\Gamma_l)^{-\frac{1}{2}}.
\end{equation}
Thus, if $m_0$ is given by Lemma \ref{lemma_flow} with the constant $\Lambda$ coming from condition \eqref{estim_energy} (which in particular controls $\dashint_{\underline{Q_i}}\vert \nabla u\vert$) there exists a distance $R\lesssim 1$ such that if $\dist(Q_i,\Gamma_l)\geq R$, then
\begin{equation}\label{eq:v0flux-inner}
    \dashint_{\underline{Q_i}} \vert \partial_d v_0\vert\leq m_0.
\end{equation}
We call these cubes ''inner cubes'' and treat separately the cubes with $\dist(Q_i,\Gamma_l)< R$ (''outer cubes''). If $Q_i$ is an inner cube, thanks to \eqref{eq:v0flux-inner}, we may apply Lemma \ref{lemma_flow} on $\underline{Q_i}$ with
\begin{equation}\nonumber
    m_i \defeq  \dashint_{\underline{Q_i}}  \partial_d v_0.
\end{equation}
This yields $\tilde u_i \in \{-1,1\}$ with $(\tilde{u}_i-u)$ compactly supported in $\underline{Q_i}$ and
\begin{eqnarray}
    \dashint_{\underline{Q_i}} \tilde{u}_i &=& \dashint_{\underline{Q_i}} (u+\partial_d v_0 ) = -\dfrac{1}{\lambda^{d-1}}\int_{\Gamma_i} (\bb + \nabla v_0)\cdot \nu,\label{eq:ui}\\
    \dashint_{\underline{Q_i}} \vert \tilde{u}_i - u\vert &\lesssim& \vert m_i\vert \overset{\eqref{eq:v0flux}}{\lesssim} \dist(Q_i,\Gamma_l)^{-\frac{1}{2}},\label{eq:uimi}\\
    \int_{\underline{Q_i}} \vert \nabla \tilde{u}_i\vert - \int_{\underline{Q_i}} \vert \nabla u \vert&\lesssim& \vert m_i\vert \overset{\eqref{eq:v0flux}}{\lesssim}\dist(Q_i,\Gamma_l)^{-\frac{1}{2}}\label{eq:varuimi}.
\end{eqnarray}
As $\tilde{u}_i=u$ near the $(d-2)$-dimensional boundary of $\underline{Q_i}$, there is no added interface at the junction between the inner cubes.
We then introduce a harmonic building block $\bb_i$ in $Q_i$ with zero boundary flux on $\Gamma_i$ and flux boundary data $\tilde{u}_i - u -\partial_d v_0$ on $\underline{Q_i}$. By Lemma~\ref{harm_building_block} for $p=2$, which we may apply because of \eqref{eq:ui}, the corresponding energy is controlled  as
\begin{equation*}
    \int_{Q_i} \vert \bb_i \vert^2 \lesssim \lambda \int_{\underline{Q_i}} \vert \tilde{u}_i - u -\partial_d v_0 \vert^2
    \lesssim \int_{\underline{Q_i}}\vert \tilde{u}_i-u \vert^2 + \int_{\underline{Q_i}}\vert \partial_d v_0 \vert^2
   \overset{\eqref{eq:uimi},\eqref{eq:v0flux}}{\lesssim} \dist(Q_i,\Gamma_l)^{-\frac{1}{2}}.
\end{equation*}

\medskip

It remains to treat the outer cubes; in this case, we can  modify the field more crudely as there are only about $l^{d-2}$ such cubes. 
If $Q_i$ is an outer cube, partition $\underline{Q_i}$ into two boxes, and define $\tilde{u}_i$ on $\underline{Q_i}$ such that $\tilde{u}_i =1$ on one box and $\tilde{u}_i=-1$ on the other. The size of the two boxes is chosen so that
\begin{equation}\label{zeroflux_u_i}
\int_{\underline{Q_i}} \left (\tilde{u}_i-\partial_d v_0 - u\right )=0.
\end{equation}
The interfacial energy in $\underline{Q_i}$ is given by
\begin{equation}\label{eq:uitilde}
\int_{\underline{Q_i}} \vert \nabla \tilde{u}_i\vert = \lambda^{d-2}\lesssim 1.
\end{equation}
We then add a harmonic building block $\bb_i$ with zero boundary flux on $\Gamma_i$ and flux boundary data $-\partial_d v_0 - u + \tilde{u}_i$ on $\underline{Q_i}$. Fix an exponent $p <2$ for which Lemma \ref{harm_building_block} holds, it is applicable because of \eqref{zeroflux_u_i}. The energy of $\bb_i$ is controlled as
\begin{equation}\nonumber
\int_{Q_i} \vert \bb_i\vert^2  \lesssim \lambda^{d-(d-1)\frac{2}{p}} \left ( \int_{\underline{Q_i}} \vert \partial_d v_0 -u+\tilde{u}_i\vert ^p\right )^{\frac{2}{p}}
\lesssim\left ( \left(\int_{\underline{Q_i}} \vert \partial_d v_0\vert ^p\right) +1 \right )^{\frac{2}{p}}.
\end{equation}
By Lemma \ref{pointwise_estimate} and since $p<2$, we obtain
\begin{equation}\label{eq:b-outer}
\int_{Q_i} \vert \bb_i\vert^2 \lesssim 1.
\end{equation}
We note that the total interfacial energy of the union of the outer cubes is controlled by $l^{d-2}$. 

\medskip

We thus define a candidate in $(\bb^*,u^*)\in \mathcal{A}^0(Q_l)$ by
 \begin{eqnarray}
     \bb^* &=& \begin{cases}
     \bb + \nabla v_0 &\text{ if } x_d\geq \lambda,\\
     \bb + \nabla v_0 + \bb_i &\text{ in } Q_i,
     \end{cases}\label{eq:defbstar}\\
     u^* &=& \tilde{u}_i \text{ in } \underline{Q_i}.\nonumber
 \end{eqnarray}
To compute its total energy $E(\bb^*,u^*,Q_l)$, we start with the interfacial energy
 \begin{align}
     \int_{\underline{Q_l}} \vert \nabla u^*\vert &\overset{\eqref{eq:uitilde}}{\leq} \int_{\underline{Q_l}} \vert \nabla u\vert  + \sum_{Q_i\text{ inner cube}} \left (\int_{\underline{Q_i}} \vert \nabla \tilde{u}_i \vert - \int_{\underline{Q_i}} \vert \nabla u \vert \right ) + C l^{d-2} \nonumber\\
     &\overset{\eqref{eq:varuimi}}{\leq} \int_{\underline{Q_l}} \vert \nabla u \vert +  C \left (\sum_{Q_i \text{ inner cube}} \dist(Q_i,\Gamma_l)^{-\frac{1}{⅛}}+ l^{d-2}\right )\nonumber\\ 
    &\overset{\hspace{.6cm}}{\leq} \int_{\underline{Q_l}} \vert \nabla u\vert + C l^{d-\frac{3}{2}}.\label{eq:ustarcontrol}
  \end{align}
 For the field energy, by Lemma \ref{orthogonality_lemma}, we note
\begin{equation*}
     \int_{Q_l} \vert \bb^*\vert^2 \overset{\eqref{eq:defbstar}}{=} \int_{Q_l} \vert \bb + \bb_i \mathbb{1}_{Q_i} +\nabla v_0 \vert^2 = \int_{Q_l} \vert \bb + \bb_i \mathbb{1}_{Q_i} \vert^2 -\int_{Q_l} \vert \nabla v_0\vert^2.
\end{equation*}
Thus, there holds
\begin{equation*}
    \int_{Q_l} \vert \bb^*\vert^2 + \int_{Q_l} \vert \nabla v_0\vert^2 \overset{\eqref{eq:defbstar}}{=} \int_{Q_l\cap\{x_d\geq \lambda\}} \vert \bb \vert^2 + \sum_{i} \int_{Q_i} \vert \bb + \bb_i\vert^2
\end{equation*}
and thus
\begin{align}\label{eq:bstarv0b}
    \int_{Q_l} \vert \bb^*\vert^2 + \int_{Q_l} \vert \nabla v_0\vert^2 -\int_{Q_l} \vert \bb \vert^2 =\sum_{i} \int_{Q_i} \left (\vert \bb_i\vert^2 + 2 \bb\cdot \bb_i \right)\hspace{-8cm}& \nonumber \\
    &\lesssim \sum_{Q_i\text{ outer cube}} \int_{Q_i} \left (\vert \bb_i\vert^2 + \vert \bb \vert^2 \right) + \sum_{Q_i \text{ inner cube}} \left ( \int_{Q_i}\vert \bb_i\vert^2 +  \left \vert \int_{Q_i} \bb\cdot \bb_i \right\vert \right).
\end{align}

\medskip

 For the inner cubes,  we integrate the mixed terms by parts, using the facts that $\bb$ is the gradient of a H\"older-$1/2$ potential $v$ (cf. Lemma \ref{hoelder_v}) and that $\bb_i$ is divergence-free with zero boundary flux on $\Gamma_i$. More precisely, choosing a point $(x'_i,0)$ at the bottom of the inner cube $Q_i$, we get:
 \begin{align*}
     \left \vert \int_{Q_i} \bb\cdot \bb_i \right\vert &= \left \vert \int_{\underline{Q_i}} (v-v((x'_i,0)) \bb_i\cdot\nu \right\vert= \left \vert \int_{\underline{Q_i}} (v-v((x'_i,0)) (\tilde{u}_i - \partial_d v_0-u) \right\vert \\
     &\leq \sup_{x'\in \underline{Q_i}} \vert v(x',0)-v((x'_i,0))\vert \left (\int_{\underline{Q_i}}\vert \tilde u_i-u\vert +\int_{\underline{Q_i}}\vert\partial_d v_0 \vert\right )\\
     &\hspace{-.4cm} \overset{\eqref{eq:v0flux}, \eqref{eq:uimi}}{\lesssim}  \dist(Q_i,\Gamma_l)^{-\frac{1}{2}}.
\end{align*}
For the outer cubes, we appeal to \eqref{eq:b-outer} and condition \eqref{estim_energy} in the choice of $\lambda$.
 We plug this into the  estimate \eqref{eq:bstarv0b} to obtain
 \begin{equation}\nonumber
    \int_{Q_l} \vert \bb^*\vert^2 + \int_{Q_l} \vert \nabla v_0\vert^2 -\int_{Q_l} \vert \bb \vert^2  \lesssim  l^{d-2} + \sum_{Q_i\text{ inner cube}}  \dist(Q_i,\Gamma_l)^{-\frac{1}{2}}
    \lesssim l^{d-\frac{3}{2}}.
\end{equation}
Combining interfacial and field energy, cf. \eqref{eq:ustarcontrol} and the last estimate, we get the desired \eqref{eq:ubstar} in form of
 \[
    E( u^*,\bb^*, Q_l)\leq E(u,\bb,Q_l)-\int_{Q_l} \dfrac{1}{2}\vert \nabla v_0\vert^2 + C l^{d-\frac{3}{2}}.
 \]
\end{proof}

\begin{proof}[Proof of Lemma \ref{sigma_sigma_null}]
Recall that $v_0$ is the solution to the over-relaxed problem \eqref{eq:overrelaxed} in $Q_l$ with flux boundary data $\bb\cdot \nu$ on $\Gamma_l$. Consider a periodic minimizer $(u^{per},\bb^{per}) \in \mathcal{A}^{per}(Q_l)$. Let $v_0^{per}$ denote the solution to the over-relaxed problem in $Q_l$ with flux boundary data $\bb^{per}\cdot \nu$. Notice that the field $\bb^{per}+\nabla v_0^{per} -\nabla v_0$ has the same normal flux boundary conditions as $\bb$ on $\Gamma_l$. We will modify $(u^{per},\bb^{per}+\nabla v_0^{per} -\nabla v_0)$ in order to obtain a candidate $(u^*,\bb^*)$ in $\mathcal{A}^{\bb\cdot\nu}(Q_l)$ such that
\begin{equation}
  \label{eq:bstarbper}
  \dfrac{E(u^*,\bb^*,Q_l)}{l^{d-1}} \leq \dfrac{E(u^{per},\bb^{per},Q_l)}{l^{d-1}} + \dfrac{1}{l^{d-1}}\int_{Q_l}\dfrac{1}{2}\vert \nabla v_0 \vert^2 + \dfrac{C}{l^{1/2}},
\end{equation}
which implies \eqref{eq:E_le_sigper} up to an additional interfacial energy of order $l^{d-2}$.
As in Lemma \ref{sigma_null_sigma}, we decompose the bottom of $Q_l$ into cubes $Q_i$ and consider separately the cubes near $\Gamma_l$ and the others. We choose cubes of side length $\lambda \sim 1$, where $\lambda$ divides $l$ and is such that on each cube $Q_\lambda$ of side length $\lambda$ contained in $Q_l$ and with base in $\underline{Q_l}$, there holds $ E(u^{\per},\bb^{\per},Q_\lambda) \lesssim 1$ ( by Theorem \ref{uniform_bound}), $\vert \dashint_{\underline{Q_\lambda}}  u^{\per} \vert \leq \frac{1}{2}$ (by Lemma \ref{control_u_avg}), and $\dashint_{\underline{Q_\lambda}}\vert \partial_d v_0^{\per}\vert \leq \frac{1}{4}$ and  $\dashint_{\underline{Q_\lambda}} \vert \partial_d v_0\vert \leq \frac{1}{4}$ (by Lemma \ref{pointwise_estimate}). Lemma \ref{pointwise_estimate} in facts yields the following more precise estimate for an inner cube $Q_i$
\begin{equation}\nonumber
 \dashint_{\underline{Q_i}} (\vert \partial_d v_0^{per}\vert + \vert \partial_d v_0\vert) \lesssim \dist(Q_i,\Gamma_l)^{-\frac{1}{2}}.
\end{equation}

\medskip

We do not detail the construction of the new charge $u^*$ (equal to $\tilde{u}_i$ on the cubes $Q_i$) as it is similar to what we did in the proof of Lemma \ref{sigma_null_sigma}, replacing $u$ by $u^{per}$ when applying Lemma \ref{lemma_flow} for the inner cubes with
\begin{equation}\nonumber
   m_i \defeq \dashint_{\underline{Q_i}} (-\partial_d v_0^{per}-\partial_d v_0).
\end{equation}
When it comes to the construction of $\bb^*$ there is a slight difference coming from the fact that we need to take both $\nabla v_0$ and $\nabla v_0^{per}$ into account. On a cube $Q_i$, we let $\bb_i$ be the harmonic building block corresponding to Lemma \ref{harm_building_block} with zero boundary flux  on $\Gamma_i$ and flux boundary data $\tilde{u}_i - u^{per} +\partial_d v_0^{per} -\partial_d v_0$ on $\underline{Q_i}$. We then define $(u^*,\bb^*)\in \mathcal{A}^{\bb\cdot \nu}(Q_L)$ by
 \begin{eqnarray}
     \bb^* &=& \begin{cases}
     \bb^{per} + \nabla v_0^{per} -\nabla v_0 &\text{ if } x_d\geq \lambda,\\
     \bb^{per} + \nabla v_0^{per} - \nabla v_0 + \bb_i &\text{ if } (x',x_d)\in Q_i,
     \end{cases}\label{eq:defbstar2}\\
     u^* &=& \tilde{u}_i \quad\text{ in } \underline{Q_i}.\nonumber
 \end{eqnarray}
As for \eqref{eq:ustarcontrol}, the interfacial energy is controlled as 
\begin{equation}\label{eq:ustar_uper}
     \int_{\underline{Q_l}} \vert \nabla u^*\vert \leq \int_{\underline{Q_l}} \vert \nabla u^{per}\vert + C l^{d-\frac{3}{2}}.
\end{equation} 
 To compute the energy of $\bb^*$, we first notice that Lemma \ref{orthogonality_lemma} implies
\begin{eqnarray*}
      \int_{Q_l} \vert \bb^*\vert^2-\int_{Q_l} \vert \nabla v_0\vert^2 &=& \int_{Q_l} \vert \bb^*+ \nabla v_0\vert^2 \\
&\overset{\eqref{eq:defbstar2}}{=}&  \int_{Q_l\cap\{x_d\geq \lambda\}} \vert \bb^{per}+\nabla v_0^{per} \vert^2 + \sum_{i} \int_{Q_i} \vert \bb^{per} + \nabla v_0^{per} + \bb_i\vert^2\\
&=& \int_{Q_l} \vert \bb^{per}+\nabla v_0^{per} \vert^2 + \sum_{i} \int_{Q_i}\left (  2 (\bb^{per} + \nabla v_0^{per}) \cdot \bb_i + \vert \bb_i\vert^2\right ).
\end{eqnarray*}
We apply Lemma \ref{orthogonality_lemma} to $\bb^{per}+ \nabla v_0^{per}$ and use the fact that $\int_{Q_l} \vert \nabla v_0^{per}\vert^2 \geq 0$ to obtain
\begin{eqnarray}
   \int_{Q_l} \vert \bb^*\vert^2-\int_{Q_l} \vert \nabla v_0\vert^2 - \int_{Q_l} \vert \bb^{per}\vert^2 &\leq& \sum_{i} \left \vert \int_{Q_i}  2 (\bb^{per} + \nabla v_0^{per}) \cdot \bb_i\right \vert + \sum_{i}\int_{Q_i} \vert \bb_i\vert^2\nonumber
\\
   &&\hspace{-3cm}\lesssim \int_{Q_l} \vert \nabla v_0^{per} \vert^2 + \sum_{Q_i \text{ outer cube}} \int_{Q_i}\left( \vert \bb_i\vert^2 + \vert \bb^{per}\vert^2\right)\nonumber\\
&+&
\sum_{Q_i \text{ inner cube}}\left ( \int_{Q_i} \vert \bb_i\vert^2  + \left \vert \int_{Q_i}  \bb^{per} \cdot \bb_i \right \vert \right).\label{eq:bstar_last}
\end{eqnarray}
Applying Lemma \ref{sigma_null_sigma} to $(u^{per},\bb^{per})$ (which is a minimizer in $\mathcal{A}^{per}(Q_l)$) and using Corollary \ref{existence_of_limit} (which only relied on Lemma \ref{sigma_null_sigma}) to get $\vert \sigma^{per}(Q_l)-\sigma^0(Q_l)\vert \lesssim l^{-1/2}$, we obtain
\[
   \dfrac{1}{l^{d-1}}\int_{Q_l} \dfrac{1}{2}\vert \nabla v_0^{per} \vert^2 
   \leq \sigma^{per}(Q_l)- \sigma^{0}(Q_l) +\dfrac{C}{l^{1/2}} \lesssim  \dfrac{1}{l^{1/2}}.
\]
The remaining terms on the r.\ h.\ s. of \eqref{eq:bstar_last} are controlled by the same arguments as those in \eqref{eq:bstarv0b}, to the effect of 
\[
   \int_{Q_l} \vert \bb^*\vert^2 - \int_{Q_l}\vert \nabla v_0\vert^2 - \int_{Q_l} \vert \bb^{per} \vert^2 \lesssim l^{d-\frac{3}{2}}.
\]
Combined with \eqref{eq:ustar_uper}, this yields \eqref{eq:bstarbper}.
\end{proof}
\begin{proof}[Proof of Corollary \ref{existence_of_limit}]
Using statements \eqref{basic per free} and \eqref{sigma_basic_comparison} of Lemma \ref{basic} as well as Lemma \ref{sigma_null_sigma} for $l=L\gg 1$, we obtain
\begin{equation}\label{eq:comparison1}
    \sigma^{per}(Q_L) \leq \sigma(Q_L) +\dfrac{C}{L}\leq \sigma^0(Q_L)+ \dfrac{C}{L} \leq \sigma^{per}(Q_L) +\dfrac{C}{L^{1/2}}.
\end{equation}
Thus, it suffices to prove that there exists $\sigma^*\in (0,+\infty)$ with $\vert \sigma^0(Q_L) -\sigma^*\vert \lesssim L^{-1/2}$. In fact, it even suffices to prove the existence of the limit $\sigma^*$ of $\sigma^0(Q_L)$ and thus of $\sigma(Q_L)$ as $L\uparrow + \infty$. Indeed, supposing the limit exists,  statements \eqref{basic free kl}, \eqref{basic 0 kl}, and \eqref{sigma_basic_comparison} of Lemma~\ref{basic}, together with \eqref{eq:comparison1} imply
\[
      \sigma(Q_L)\leq \lim_{k\to \infty} \sigma(Q_{kL}) = \sigma^* = \lim_{k\to \infty} \sigma^0(Q_{kL}) \leq \sigma^0(Q_L) \leq \sigma(Q_L)+ \dfrac{C}{L^{1/2}}.
\]
We now argue that $\lim_{L\to +\infty} \sigma^0(Q_L)\in (0,+\infty)$ by showing that the integer monotonicity of Lemma \ref{basic} \eqref{basic 0 kl} approximately extends to all $L\gg l \gg 1$ in form of
\begin{equation}\label{eq:sigma0monotone}
  \sigma^0(Q_L) \leq \sigma^0(Q_l) + \dfrac{C}{l}.
\end{equation}
This together with Lemma \ref{basic} \eqref{sigma_basic_comparison} and \eqref{basic lower bound} yield the existence of a positive and finite limit. 
In order to prove \eqref{eq:sigma0monotone}, let $(u,\bb)$ be a minimizer in $\mathcal{A}^0(Q_L)$ for $L\gg 1$. Given a positive number $\lambda$, we define $(u^\lambda,\bb^\lambda)$ by $u^\lambda(\lambda\, \cdot)=  u( \cdot)$ and $\bb^\lambda(\lambda\, \cdot) = \bb( \cdot)$  (as in the proof of Theorem \ref{equipartition}).
Clearly $(u^\lambda,\bb^\lambda)\in \mathcal{A}^0(Q_{\lambda L})$, and it is easy to see by a change of variables that
\[
    \dfrac{E(u^\lambda,\bb^\lambda, Q_{\lambda L})}{(\lambda L)^{d-1}} \leq \max\Big\{\lambda, \dfrac{1}{\lambda}\Big\} \dfrac{ E(u,\bb,Q_L)}{L^{d-1}}.
\]
From this we infer that for $L\geq L'\gg 1$,  $\sigma^0(Q_L) \leq  (L/L')\,  \sigma^0(Q_{L'})$ 
and in turn, as Lemma~\ref{basic}~\eqref{sigma_0_basic_bound} implies $\sigma^0(Q_L)\lesssim 1$, there holds
\begin{equation}\label{eq:densitydiff}
        \vert \sigma^0(Q_{L'}) - \sigma^0(Q_{L}) \vert \lesssim \dfrac{L}{L'} -1.
\end{equation}
Now, for $L\gg l \gg 1$, let $k\in \N$ be  such that $(k+1) l > L \geq k l$; applying \eqref{eq:densitydiff} with $L' = kL$, we obtain by Lemma \ref{basic} \eqref{basic 0 kl}
\[
    \sigma^0(Q_L) \leq \sigma^0(Q_{kl}) + \dfrac{C}{l} \leq \sigma^0 (Q_{l})+ \dfrac{C}{l}.
\]
\end{proof}
\begin{proof}[Proof of Lemma \ref{control-v_0}]
 Let $1\ll l' \ll l\leq L$, where we assume $l$ to be an odd integer multiple of $l'$. Let $-\nabla v_0^{l'}$ be the solution of the over-relaxed problem \eqref{eq:overrelaxed} in $Q_{l'}$ with flux boundary data $\bb\cdot \nu$ on $\Gamma_{l'}$. We first claim that
  \begin{equation}\label{eq:diffv0}
    \int_{Q_{l'}} \vert \nabla v_0^{l'} - \nabla v_0^{l}\vert^2 \lesssim l^{d-1}l'^{\ \frac{1}{2}}.    
  \end{equation}
To see this, we decompose $Q_l\cap \{x_d <l'\}$ into cubes $Q^i$ of side $l'$ (including $Q_{l'}$), and denote by $-\nabla v_0^{i}$ the corresponding solutions of the over-relaxed problem. Then the field
\begin{equation}\nonumber
  \tilde{\bb} \defeq \begin{cases}
    -\nabla v_0^i &\text{ on }Q^i,\\
     \ \ \ \bb &\text{ on } Q_l \cap \{x_d\geq l'\},
  \end{cases}
\end{equation}
is a competitor for the over-relaxed problem on $Q_l$, so that, using Lemma \ref{orthogonality_lemma},
\begin{eqnarray*}
  \int_{Q_l} \dfrac{1}{2} \vert \tilde{\bb} + \nabla v_0^l\vert^2 &\leq& \int_{Q_l} \dfrac{1}{2} \vert \tilde{\bb}\vert^2  - \int_{Q_l} \dfrac{1}{2} \vert \nabla v_0^l\vert^2\\
  &\leq& \int_{Q_l\cap \{x_d >l'\}} \dfrac{1}{2}  \vert \bb \vert^2 + \sum_i \int_{Q^i} \dfrac{1}{2} \vert \nabla v_0^i\vert^2 - \int_{Q_l} \dfrac{1}{2} \vert \nabla v_0^l\vert^2.
\end{eqnarray*}
On the one hand, we know by Lemma \ref{sigma_null_sigma} and Corollary \ref{existence_of_limit} that
\begin{equation}\nonumber
   \int_{Q^i} \dfrac{1}{2} \vert \nabla v_0^i\vert^2 \leq E(u,\bb,Q^i) -\sigma^* l'^{\ d-1} + C l'^{\ d-\frac{3}{2}}.
\end{equation}
On the other hand, by Lemma \ref{sigma_sigma_null} and Corollary \ref{existence_of_limit}, we have
\begin{equation}\nonumber
   -\int_{Q_l} \dfrac{1}{2} \vert \nabla v_0^l\vert^2 \leq E(u,\bb,Q_l) -\sigma^* l^{d-1} + C l^{d-\frac{3}{2}}.
\end{equation}
Since the energy is additive up to the interfacial energy coming from pasting, 
\begin{equation}\nonumber
   E(u,\bb,Q_l) \leq \int_{Q_l\cap \{x_d >l'\}} \dfrac{1}{2}  \vert \bb \vert^2 + \sum_i  E(u,\bb,Q^i)+ C \dfrac{l}{l'}l^{d-2},
\end{equation}
we obtain
\begin{equation}\nonumber
  \int_{Q_l} \dfrac{1}{2} \vert \tilde{\bb} + \nabla v_0^l\vert^2 \lesssim \left (\dfrac{l}{l'}\right)^{d-1} l'^{\ d-\frac{3}{2}}+ \dfrac{l}{l'}l^{d-2} \lesssim \dfrac{l^{d-1}}{l'^{\ 1/2}}.
\end{equation}
Restricting the integral to $Q_{l'}$, we recover \eqref{eq:diffv0} by definition of $\tilde{\bb}$.

\medskip

Second, we claim
 \begin{equation}\label{eq:v0lonQlprime}
    \int_{Q_{l'}} \vert \nabla v_0^{l}\vert^2 \lesssim \left(\dfrac{l'}{l}\right )^{d} \int_{Q_l} \vert \nabla v_0^{l}\vert^2.
  \end{equation}
Indeed, reflecting $v_0^l$ oddly across $\{x_d=0\}$, we obtain a harmonic function on the box $(-l/2,l/2)^{d-1}\times(-l,l)$, so that \eqref{eq:v0lonQlprime} is a consequence of the mean-value property of the sub-harmonic function $\vert \nabla v_0^l\vert^2$.

\medskip

Combining \eqref{eq:diffv0} and \eqref{eq:v0lonQlprime} with the triangle inequality, and letting $\theta \defeq l'/l$ we get
\begin{equation}\nonumber
    \dfrac{1}{(\theta l)^{d-1}} \int_{Q_{\theta l}}\dfrac{1}{2} \vert \nabla v_0^{\theta l}\vert^2 \leq C_0 \left (\theta \dfrac{1}{l^{d-1}} \int_{Q_l} \dfrac{1}{2}\vert \nabla v_0^l\vert^2 + \dfrac{1}{(\theta l)^{1/2}}\right ).
\end{equation}
We fix $\theta$, depending only on $d$, such that $1/\theta$ is an odd integer, and so small that $C_0 \theta \leq 1/2$; so that letting $D_l \defeq l^{-(d-1)} \int_{Q_l} \frac{1}{2} \vert \nabla v_0^l\vert^2$, we obtain
\begin{equation}\label{eq:stepv0}
  D_{\theta l} \leq \dfrac{1}{2} D_l + \dfrac{C}{l^{1/2}}.
\end{equation}
Applying Lemma \ref{sigma_null_sigma} and Corollary \ref{existence_of_limit} to $l=L$, we get
\begin{equation}\nonumber
  D_L \leq \dfrac{E(u,\bb,Q_L)}{L^{d-1}}  -\sigma^0(Q_L) + \dfrac{C}{L^{1/2}} \lesssim \dfrac{1}{L^{1/2}}.
\end{equation}
With this anchoring, we iterate \eqref{eq:stepv0}, to obtain the desired $D_l\lesssim l^{-1/2}$ for $l\gg 1$.
\end{proof}

\begin{proof}[Proof of Lemma \ref{lemma_flow}]
 By rescaling, we may assume without loss of generality that $\lambda=1$, so that $\cube = (0,1)^{d-1}$. We will define a smooth  vector field $\xi$, compactly supported in~$\cube$, 
such that $\Div \xi$ approximates $u- \dashint_\cube u$. Then we will transport
the values of $u$ along the flow generated by $\xi$, for positive or negative time depending on whether we want to increase or decrease the average charge. The flow of $\xi$ is the map $\Phi_{\cdot}(\cdot): (t,x)\in \R \times \cube \to \cube$ solving the differential equation
\begin{equation}\nonumber
\begin{cases}    
\partial_t \Phi_t(x)= \xi (\Phi_t(x))\quad &\textrm{for } (t,x)\in \R\times \cube, \\
                  \ \ \Phi_t \vert_{t=0}= \Phi_0 = \mathrm{id} \quad &\textrm{on } \cube.
\end{cases}
\end{equation}
As $\xi$ will be smooth and compactly supported, for all $t\in \R$, $\Phi_t$ is a diffeomorphism of $\cube$ that coincides with the identity close to the boundary of $\cube$. For shortness, we set
\begin{equation}\nonumber
u_t:=u\circ \Phi_t^{-1}.     
\end{equation}
By definition, $u_t$, like $u$, takes values into $\{-1,1\}$. We shall construct $\xi$ such that
\begin{equation}\label{int_u_t_grows}
     \frac{d}{dt}\int_\cube u_t \geq \frac{1}{4} \quad \text{ for } \vert t\vert \ll 1
\end{equation}
and that 
$\int_\cube |u_t-u|$
and $\int_\cube |\nabla u_t| $ are Lipschitz continuous at $t=0$. This will allow us
to define $\tilde{u} \defeq u_t$ for some appropriate $t \in [\,-4m, 4m\,]$ with the desired properties.

\medskip

We now turn to the construction of the vector field $\xi$. In order to flow from positive to negative charges, we want $\nabla \cdot \xi\simeq u-\dashint u$. In particular, a good candidate is the gradient of the potential $\psi$ obtained by solving the following Poisson problem
\begin{equation}\label{eq:poisson}
    \begin{cases}-\Delta \psi &= u-\dashint_\cube u  \hspace{.8cm}\text{in } \cube, \\
    -\nabla \psi \cdot \nu &= 0  \hspace{2cm} \text{on } \partial (\cube). 
    \end{cases}
\end{equation}
Notice that the vector field $- \nabla \psi$ has divergence equal to $u - \dashint_\cube u$.
We reflect $u$ and $\psi$ evenly along the sides of $\cube$ to extend them to the whole of $\R^{d-1}$. The extended $\psi$ still satisfies $-\Delta \psi = u-\dashint_\cube u$. Since $ \sup_{\R^{d-1}}\vert u-\dashint_\cube u \vert \leq 2$, by elliptic regularity (we refer to Section 8.11 in \cite{GilbargTrudinger}),  $\nabla\psi$ is H\"older-$1/2$ continuous. In particular, it is square integrable in the larger cube $\tilde{\cube} \defeq (-1,2)^{d-1}$. We will use the following uniform bounds:
\begin{equation} \label{psi_estimate}
\forall x,y \in \R^{d-1},\, \vert \nabla \psi (y)-\nabla \psi (x)\vert  \lesssim \vert y-x\vert^{\frac{1}{2}}  \quad \text{and} \quad
\int_{\tilde{\cube}}\vert \nabla \psi \vert^2 \lesssim 1.
\end{equation}
 To obtain $\xi$, we cut off and mollify $-\nabla \psi$ to obtain a smooth vector field, compactly supported in $\cube$. 
Given $r\in(0,1/2\,]$, for any function $f$, we denote by $f^r := f * \varphi_r$ the convolution of $f$
with a standard mollifier $\varphi_r$ on the scale $r$. 
Next, fix $\eta^1_r:(0,1)\to [\,0,1\,]$
to be a smooth compactly supported cut-off function such that
$\eta^1_r (s)=1$ for $s\in [r,1-r]$ ,
and $|(\eta^{1}_r)'|\lesssim \frac{1}{r}$. We now define a cut-off function $\eta_r$ on $\cube$ for $x=(x_1,\dots, x_{d-1})\in \cube$ by $\eta_r (x):= \Pi_{i=1}^{d-1} \eta^1_r(x_i)$. Finally we define $\xi$ on $\cube$ by 
\begin{equation}\label{eq:defxi}
\xi := -\eta_r \nabla \psi^r.
\end{equation}
 We  do not stress the dependence of $\xi$ on $r$, as $r$ will later be fixed. There holds 
\begin{equation}\label{div_xi}
    \Div \xi = - \eta_r \Delta \psi^r - \nabla \eta_r \cdot \nabla \psi^r
    \overset{\eqref{eq:poisson}}{=} \eta_r \big(u^r - \dashint_\cube u \big) -  \nabla \eta_r \cdot \nabla \psi^r. 
\end{equation}

\medskip

Let us show that the second term $\nabla \eta_r \cdot \nabla \psi^r$ is small in the $\mathrm{L}^1$-norm, for small $r$. 
Notice that $\nabla \eta_r$ is supported in the set of points lying at distance less than $r$ to $\partial( \cube)$. Consider one of these points $x$. Without loss of generality, we may suppose that there exists $k\in \{1,\dots,d-2\}$ such that $x_i\in (0,r)$ for $i\leq k$ and $x_i \in [\,r,1-r\,]$ for $i >k$. As $x_i\in [\,r,1-r\,]$ implies $(\eta^1_r)'(x_i)= 0$, there holds
\begin{align*}
  \nabla \eta_r \cdot \nabla \psi^r(x) =  \sum_{i= 1}^{k} \partial_i \eta_r (x)  \partial_i\psi^r(x).
\end{align*}
By estimate  \eqref{psi_estimate} and the Neumann boundary condition in \eqref{eq:poisson}, for $i\leq k$, we have
\[
     \vert \partial_i\psi^r(x)\vert \lesssim r^{\frac{1}{2}}.
\]
Together with $|\nabla \eta_r| \lesssim 1/r$, we thus obtain
\begin{equation}\nonumber
 \vert  \nabla \eta_r (x)\cdot \nabla \psi^r(x)\vert \lesssim \dfrac{1}{r^{1/2}}.
\end{equation}
The set of all such $x$ has area of order $r$. We thus have
\begin{equation}\label{product_L1}
  \int_{\cube} \vert \nabla \eta_r \cdot \nabla \psi^r\vert \lesssim r^{\frac{1}{2}}.
\end{equation}

\medskip

On the one hand, by the convolution estimate and recalling the extension of $u$ by even reflection,
\begin{equation}\label{uurnablau}
  \int_{\cube}\vert u^r-u\vert \lesssim r \int_{\cube} \vert \nabla u\vert \leq r \Lambda
\end{equation}
and on the other hand, as $\vert u\vert \leq 1$,
\begin{equation}\label{eta-1u}
  \int_{\cube} \Big \vert (\eta_r-1) \big(u^r-{\dashint}_\cube u\big )\Big\vert \lesssim r.
\end{equation}
We infer
\begin{align}
    \int_{\cube} \Big\vert \Div \xi -\big(u-\dashint_\cube u\big)\Big\vert &\overset{\eqref{div_xi}}{\leq}\int_{\cube} \Big\vert \eta_r \big(u^r - \dashint_\cube u \big) -  \nabla \eta_r \cdot \nabla \psi^r -\big(u-\dashint_\cube u\big)\Big\vert
\nonumber \\
&\hspace{-1.5cm}\leq \int_{\cube} \vert u^r-u\vert + \int_{\cube} \Big\vert (\eta_r-1)\big(u^r - \dashint_\cube u \big)\Big \vert +  \int_{\cube}\vert \nabla \eta_r \cdot \nabla \psi^r\vert 
\overset{\eqref{uurnablau},\eqref{eta-1u},\eqref{product_L1}}{\lesssim_\Lambda} r^{\frac{1}{2}}. \label{B_L_1_est}
\end{align}

\medskip

Now that we have this control over $\xi$, let us prove that its flow modifies the global charge in the desired way \eqref{int_u_t_grows}. By a change of variables,
\begin{equation*}
    \int_\cube u_t = \int_\cube u \det \mathrm{D}\Phi_t.
\end{equation*}
Using Liouville's formula to differentiate the determinant 
and operating the converse change of variables yields
\begin{align} \label{d_t_u_t}
    \frac{d}{dt}\int_\cube u_t = \int_\cube  u (\Div \xi )\circ \Phi_t \det \mathrm{D}\Phi_t
    = \int_\cube u_t \Div \xi.
\end{align}
Hence at $t=0$, using \eqref{B_L_1_est} and the fact that $|u| =1$, we obtain
\begin{equation*}
   \left \vert \left. \Frac{d}{dt}\right\vert_{t=0}\int_\cube u_t - \int_\cube u \big(u- \dashint_\cube u \big) \right \vert \lesssim_\Lambda r^{\frac{1}{2}},
\end{equation*}
so together with the assumption \eqref{control_avg}, we get 
\begin{equation*}
  \dfrac{3}{4} -  \left.\Frac{d}{dt}\right\vert_{t=0}\int_\cube u_t \lesssim_\Lambda r^{\frac{1}{2}}.
\end{equation*}
Hence we now may fix $r>0$ so small that 
\[
    \left.\Frac{d}{dt}\right\vert_{t=0}\int_\cube u_t \geq \dfrac{1}{2}.
\]

It remains to prove that \eqref{int_u_t_grows} holds also for $t$ small enough and not just at $t=0$. We postpone this and start by proving that the total variation of $u_t$ is uniformly bounded. By \cite[Theorem 17.5]{Maggi_Book}, or by a standard generalization
of \cite[Theorem 10.4]{Giusti_Book}, the first variation of the total variation of $u_t$ at time $t$ along the flow of $\xi$ is equal to 
\begin{equation*}
  \frac{d}{dt}\int_\cube |\nabla u_t| = -\int_\cube \left (\nabla_{tan} \cdot \xi \right ) |\nabla u_t|,
\end{equation*}
where $\nabla_{tan} \cdot \xi$ is the tangential divergence  of $\xi$ along the reduced boundary of the set of finite perimeter $\{u_t=-1\}$.
Since $r$ is now fixed and depends only and $\Lambda$ and $d$, we obtain from \eqref{eq:defxi}, the fact that $\psi^r = \psi * \varphi_r$ and the H\"older inequality that
\begin{equation}
  \label{eq:boundxi}
  \sup_\cube \vert \xi\vert \lesssim_\Lambda\  \Big ( \sup_\cube \vert \nabla \psi\vert^2\Big)^{\frac{1}{2}}\  \overset{\eqref{psi_estimate}}{\lesssim_\Lambda}\  1 \quad \text{and} \quad
    \sup_\cube \vert \mathrm{D}\xi \vert\lesssim_\Lambda 1.
\end{equation}
  Hence  
\begin{equation*}
\Big \vert  \frac{d}{dt}\int_\cube|\nabla u_t| \Big \vert\lesssim_\Lambda \int_\cube|\nabla u_t|
\end{equation*}
and thus for $t\in[\,-1,1\,]$
\begin{equation}\label{eq:varinterface}
\int_\cube|\nabla u_t| - \int_\cube|\nabla u|\lesssim_\Lambda t.
\end{equation}

\medskip

Let us now go back to proving \eqref{int_u_t_grows} for small $t$. It is enough to prove that the function $ t\mapsto \frac{d}{dt}\int_\cube u_t$ is Lipschitz continuous at $t=0$, 
with a Lipschitz constant depending only on the total variation of $u$ (the bound $\Lambda$ from the statement). Indeed, by \eqref{d_t_u_t},
\begin{equation}\nonumber
 \Big|  \frac{d}{dt}\int_\cube u_t  - \frac{d}{dt}\Big\vert_{t=0}\int_\cube u_t  \Big| 
 \leq \int_\cube |u_s-u| |\Div \xi|
\overset{\eqref{eq:boundxi}}{\lesssim_\Lambda} \int_\cube |u_s-u| .
\end{equation}
Hence $t \mapsto \frac{d}{dt}\vert\int_\cube u_t\vert $  is Lipschitz continuous
at $t=0$ provided the function $t\mapsto\int_\cube |u_t-u|$ is as well. To this purpose, we make use of
\begin{equation}\nonumber
 \int_\cube |u_s-u| = \sup_{\zeta\in C^{\infty}_0(\cube), \vert \zeta \vert \leq 1} \int_\cube  (u_s-u)\zeta.
\end{equation}
In particular, it suffices to show that the functions $t\mapsto \int_\cube  (u_t-u)\zeta$ are uniformly Lipschitz continuous at $t=0$. For this, we note that by a similar argument as for \eqref{d_t_u_t}, we have
\[
    \frac{d}{dt} \int_\cube (u_t-u) \zeta = \frac{d}{dt} \int_\cube u_t \zeta  = \int_\cube u_t \Div(\zeta \xi)
\]
and thus
\[
   \Big \vert \frac{d}{dt} \int_\cube (u_t-u) \zeta\Big \vert \leq \sup_\cube \vert \zeta \xi\vert \int_\cube \vert \nabla u_t\vert \overset{\eqref{eq:varinterface},\eqref{control_perim}}{\lesssim_\Lambda} \sup_\cube \vert \xi \vert \ \overset{\eqref{eq:boundxi}}{\lesssim_\Lambda} \ 1.
\]
\end{proof}
\begin{proof}[Proof of Lemma \ref{control_u_avg}]
Let us choose a smooth cut-off function $\eta$ compactly supported in $Q_\lambda \cup \underline{Q_\lambda}$ and such that 
 $\dashint_{\underline{Q_\lambda}}(1-\eta) \leq 1/4$, while $\vert \nabla \eta\vert \lesssim 1/\lambda$. On the one hand, this implies
\[
   \Big \vert \dashint_{\underline{Q_\lambda}} u\Big \vert  \leq \Big \vert \dashint_{\underline{Q_\lambda}}\eta u  \Big \vert + \frac{1}{4},
\]
so that it is enough to establish
\[
   \Big \vert \dashint_{\underline{Q_\lambda}}\eta u  \Big \vert \leq \frac{1}{4}.
\]
On the other hand we obtain from $\eqref{eq:constraints}$ 
\[
    \int_{\underline{Q_\lambda}}\eta u 
= -\int_{Q_\lambda} \nabla \eta\cdot \bb
\]
and as $\vert \nabla \eta\vert \lesssim 1/\lambda$, by the H\"older inequality and Theorem \ref{uniform_bound}, we have
\[
    \Big \vert \dashint_{\underline{Q_\lambda}}\eta u  \Big \vert \lesssim \left ( \dfrac{1}{\lambda^d} \int_{Q_\lambda} \vert \bb \vert^2 \right )^{1/2}\lesssim \dfrac{1}{\lambda^{1/2}},
\]
so that the conclusion follows provided $\lambda \gg 1$.
\end{proof}

\section*{Acknowledgements} We thank Tobias Ried for helpful comments on the manuscript and Emanuele Spadaro for suggesting the reference \cite{EspoFusc2011remark}. A.~J.~gratefully acknowledges the hospitality of the Max Planck Institute in Leipzig.

\bibliographystyle{amsplain}
\bibliography{EnergyDistrib}

\end{document}